\documentclass[a4paper,12pt,oneside,reqno]{amsart}

\usepackage[headinclude,DIV13]{typearea}
\areaset{15.1cm}{25.0cm}
\usepackage[latin1] {inputenc}
\usepackage{txfonts}  
\usepackage{hyperref}
\usepackage{amssymb}
\usepackage{amsmath}
\usepackage{amsthm}
\usepackage{graphicx}
\usepackage{enumerate}   
\usepackage{amsfonts}
\usepackage{dsfont}
\usepackage{tikz}
\usepackage{url}
\usetikzlibrary{intersections}
\usepackage{pgfplots}

\newtheorem{theorem}{Theorem}[section]
\newtheorem{lemma}[theorem]{Lemma}
\newtheorem{proposition}[theorem]{Proposition}
\newtheorem{corollary}[theorem]{Corollary}

\theoremstyle{definition}

\theoremstyle{remark}
\newtheorem*{remark}{Remark}

\numberwithin{equation}{section}

\newcommand{\dd}{\mathtt{d}}
\newcommand{\RR}{{\mathbb R}}  
\newcommand{\NN}{{\mathbb N}}  
\newcommand{\PP}{{\mathbb P}}  
\newcommand{\EE}{{\mathbb E}}  
\newcommand{\var}{\hbox{\rm Var}}
\newcommand{\capa}{\hbox{\rm cap}}
\newcommand{\eee}{{\mathrm e}}
\newcommand{\Ss}{\mathcal{S}}

\newcommand{\si}{\sigma}
\newcommand{\ti}{\tilde}
\newcommand{\Dd}{\mathcal{D}}
\newcommand{\Qq}{\mathcal{Q}}

\newcommand{\abs}[1]{\lvert #1 \rvert}

\newcommand{\GN}{\Gamma_N}
\newcommand{\tH}{\tilde{H}_N}

\begin{document}
	

\title[Glauber dynamics  for the RDCW model]{Metastability for the dilute Curie--Weiss model with
	Glauber dynamics}
\author[A. Bovier]{Anton Bovier}
\address{A. Bovier\\ Institut f\"ur Angewandte Mathematik\\
	Rheinische Friedrich-Wilhelms-Universit\"at\\ Endenicher Allee 60\\ 53115 Bonn, Germany }
\email{bovier@uni-bonn.de}
\author[S. Marello]{Saeda Marello}
\address{S. Marello\\ Institut f\"ur Angewandte Mathematik\\
	Rheinische Friedrich-Wilhelms-Universit\"at\\ Endenicher Allee 60\\ 53115 Bonn, Germany }
\email{marello@iam.uni-bonn.de}
\author[E. Pulvirenti]{Elena Pulvirenti}
\address{E. Pulvirenti\\ Institut f\"ur Angewandte Mathematik\\
	Rheinische Friedrich-Wilhelms-Universit\"at\\ Endenicher Allee 60\\ 53115 Bonn, Germany }
\email{pulvirenti@iam.uni-bonn.de}

\date{\today}

\begin{abstract}
	We analyse  the metastable behaviour of the  dilute Curie--Weiss model
	subject to a Glauber dynamics. The model is a random version of a mean-field Ising model, 
	where the coupling coefficients are Bernoulli random variables with mean $p\in (0,1)$. This model 
	can be also viewed as an Ising model on the Erd\H{o}s--R{\'e}nyi random graph with edge probability $p$.
	The system is a Markov chain
	where spins flip according to a Metropolis dynamics at inverse temperature $\beta$. We compute
	the average time the system
	takes to reach the stable phase when it starts from a certain probability distribution on the metastable state (called 
	the last-exit biased distribution),
	in the regime where $N\to\infty$, $\beta>\beta_c=1$
	and $h$ is positive and
	small enough. We obtain asymptotic bounds on the probability 
	of the event that the mean metastable hitting time is approximated by that of the Curie--Weiss model. The proof uses the potential theoretic approach to metastability and concentration of 
	measure inequalities.	
\end{abstract}
\thanks{
	This work was partly funded by the Deutsche Forschungsgemeinschaft (DFG, German Research Foundation) 
	under Germany's Excellence Strategy - GZ 2047/1, Projekt-ID 390685813 
	and by the Deutsche Forschungsgemeinschaft (DFG, German Research Foundation) - Projektnummer 211504053 - SFB 1060. \\
	We thank Alessandra Bianchi for fruitful discussions on capacity estimates and Frank den Hollander for giving 
	us early access to his article in preparation with Oliver Jovanovski~\cite{dHJ20} and for useful comments.
}

\maketitle

	
\section{Introduction and main results}
The randomly dilute Curie--Weiss model (RDCW)  is a classical model of a disordered ferromagnet and was studied, e.g. in  Bovier and Gayrard~\cite{BG}.
It generalises  the standard Curie--Weiss  model (CW) in that the fixed interactions between each pair of spins is replaced by independent, identically distributed, random ferromagnetic couplings 
between any pair of spins.
In Bovier and Gayrard~\cite{BG} it is  proven that the RDCW free energy converges, in the thermodynamic limit, 
to that of the  CW model, 
under some assumptions on the coupling distribution. 
Their result relies on the fact that the RDCW Hamiltonian can be approximated 
by that of the CW model up to a small perturbation which can be uniformly bounded in high probability.
In the last decade the RDCW model have gained again some attention and 
various results at equilibrium have been proven, both in the annealed and
quenched case. 
De Sanctis and Guerra~\cite{DSG08} give an exact expression of 
the free energy first in the high temperature and low connectivity regime, and then 
at  zero temperature. The control of the fluctuations of the magnetisation
in the high temperature limit is addressed by De Sanctis~\cite{DSa09}, while 
recently Kabluchko, L\"{o}we and Schubert~\cite{KLS19} prove a quenched
Central Limit Theorem for the magnetisation in the high temperature
regime.

One of the features which make these random systems with ``bond disorder''  very appealing
is their deep connection with the theory of \emph{random graphs}, which attracted great interest 
in the last years due to their application to real-world networks. Indeed, if the random couplings
are chosen as  i.i.d. Bernoulli random variables with mean $p$, one can  view the model as a
spin system on an Erd\H{o}s--R{\'e}nyi random graph with \emph{fixed} edge probability $p$, 
which makes it a dense graph. 
There has been an extensive study of the Ising model at equilibrium on different kinds of random graphs, e.g. 
in Dembo, Montanari~\cite{DM10} and Dommers, Giardin\`a, van der Hofstad~\cite{DGvdH10}, where several
thermodynamic quantities were analysed when the graph size tends to infinity.
These results were all obtained for sparse graphs which have a locally tree-like structure.
We refer to van der Hofstad~\cite{vdHnotes}  for 
a general overview of these results.

In contrast to the substantial body of literature on the equilibrium properties of the RDCW model, 
much less is known about its dynamical properties. The present paper  focuses on  the phenomenon of  \emph{metastability} for the RDCW model where, for simplicity,  the couplings  are Bernoulli distributed with \emph{fixed} parameter  $p \in (0,1)$, independent of the number of vertices $N$, and 
the system evolves according to a Glauber dynamics. In particular, we give a precise 
estimate of the mean transition time from a certain probability distribution on the \emph{metastable state} 
(called the last-exit biased distribution) to the \emph{stable state},
when the external magnetic field is small enough and positive and when $N$ tends to infinity.
We obtain asymptotic bounds on the probability of the event that the 
average time is close to the CW one times some constants of order $1$ which depend 
on the parameters of the system.

In the context of metastability for interacting particle systems on random graphs,\break
progress has been made for the case of the random regular graph, analysed by Dommers~\cite{Dom17} and 
for the configuration model, studied by Dommers, den Hollander, Jovanovski, and Nardi~\cite{DdHJN17}, 
both subject to Glauber 
dynamics, in the limit as the temperature tends to zero and the number of vertices is fixed. 
Both are dealing with sparse random graphs.
In \cite{dHJ20} den Hollander and Jovanovski  investigate the same model 
considered in the present paper and obtain estimates  on the average crossover time for fixed temperature in the thermodynamic limit.
They show that, with high probability, the exponential term is the same as in the CW 
model, while the multiplicative term is polynomial in $N$. Their analysis relies on coupling arguments
and on  the \emph{pathwise approach} to metastability. This method uses large deviations techniques
in path space  and focuses on properties of typical paths in the spirit of Freidlin-Wentzell theory.
We refer to the classical book by Olivieri and Vares~\cite{olivieri_vares_2005}
for an overview on this method. 

In contrast, in the present paper, we use the \emph{potential theoretic approach} initiated by
Bovier, Eckhoff, Gayrard and Klein in a series of papers \cite{BEGK1,BEGK2,BEGK3}
(see the  monograph of Bovier and den Hollander~\cite{BdH} 
for an in-depth review of this as well as other approaches). This method gives less information on the
evolution of the system, but leads to more precise estimates of
the metastable transition time. It has been successfully applied to a large variety of systems such 
as the random field CW model, 
where the external magnetic field is given by i.i.d. random variables, first by Bovier, Eckhoff, Gayrard and Klein in \cite{BEGK1} and later by Bianchi, Bovier and Ioffe in \cite{BBI09}. 
Furthermore, inspired by the results of Bovier and Gayrard~\cite{BG}, namely that the equilibrium properties 
of the RDCW model are very close to those of the CW model, we observe that, using Talagrand's concentration inequality, the mesoscopic measure can be expressed in terms of that of CW. 

Before stating our results we give a precise definition of the model.

\subsection{Glauber dynamics for the RDCW model}\label{sec:GdER}
Let $[N]=\{1,...,N\}$, $N\in\NN$, be a set of vertices. To each vertex $i \in[N]$ an Ising spin $\sigma_i$ with values in $\{-1,+1\}$ is associated.
We denote  by $\si=\{\sigma_i\,: \, i\in[N], \sigma_i\in\{-1,+1\}\}$ a spin configuration and we define the state space $\Ss_N=\{-1,+1\}^{N}$ to be the set of all such configurations $\si$.
We fix a probability $p \in (0,1)$. Then the \emph{randomly dilute Curie--Weiss} model (RDCW) has the following  \emph{random Hamiltonian} 
$H_N: \Ss_N\to\RR$
\begin{equation}\label{eq:hamER}
	H_N(\sigma)=- \frac{1}{Np} \sum_{1\leq i < j \leq N} J_{ij} \sigma_i \sigma_j - h \sum_{i\in[N]}\sigma_i,
\end{equation}
where $h\in \RR$ represents an external constant magnetic field, while $J_{ij}/Np$ is a ferromagnetic 
random coupling. In particular, $\{J_{ij}\}_{i,j\in[N]}$ is a sequence of i.i.d. random variables 
with $J_{ij} \sim \text{ Ber}(p)$ and $J_{ij}=J_{ji}$.

Let us denote by  $\PP_J$ the joint probability distribution of the the random couplings $J_{ij}$ with $i,j \in [N]$ and by $\EE$ the corresponding mean value.

The RDCW model can be seen as  the Ising model on the 
Erd\H{o}s--R{\'e}nyi random graph with vertex set $[N]$,  edge set $E$ and  edge probability 
$p\in (0,1)$ (see van der Hofstad~\cite{vdH} for a general overview on random graphs). 
In this picture the Hamiltonian can also be written as
\begin{equation}\label{eq:erham}
	H_N(\sigma)=- \frac{1}{Np} \sum_{\{i,j\} \in E}  \sigma_i \sigma_j - h \sum_{i\in[N]}\sigma_i.
\end{equation}
The Gibbs measure associated to the random Hamiltonian $H_N$ is 
\begin{equation}\label{eq:eqmeas} 
	\mu_{\beta, N}(\si)= \frac{\eee^{-\beta H_N(\sigma)}}{Z_{\beta,N}}, \quad \si\in\Ss_N, 
\end{equation}
where $\beta\in(0,\infty)$ is the inverse temperature and  the partition function is defined as
\begin{equation}\label{eq:partf}
	Z_{\beta,N}= \sum_{\sigma\in\Ss_N} \eee^{-\beta H_N(\si)}.
\end{equation}

The Gibbs measure $\mu_{\beta,N}$ is the unique invariant (and reversible)  measure for the
(discrete time) Glauber dynamics 
on $\Ss_N$ with Metropolis transition probabilities 
\begin{equation}\label{eq:rates}
	p_N(\si,\si')= \begin{cases}
		\frac{1}{N} \exp\big(-\beta [H_N(\sigma') -H_N(\sigma) ]_+\big), &\quad \text{if } \si\sim\si',\\
		1- \sum_{\eta\neq \si} p_N(\si,\eta), &\quad \text{if } \si=\si',\\
		0, &\quad \text{else, }
	\end{cases}
\end{equation}
where $\si\sim\si'$ means $||\si-\si'||=2$ with $||\cdot||$ the $\ell_1$-norm on $\Ss_N$, i.e. $\si\sim\si'$ if and only if $\si'$ is obtained from $\si$ by a single spin flip. 
We denote  this Markov chain by $\{\si(t)\}_{t\geq0}$ and write
$\PP_\nu$  for the law of the process $\si(t)$ 
with initial distribution $\nu$ conditioned on the realisation of the random couplings. 
Analogously, $\EE_\nu$ is the quenched  expectation with respect to the Markov
chain with initial distribution $\nu$. Moreover, 
we set $\PP_\si= \PP_{\delta_\si}$.  
For any subset $A\subset \Ss_N$ we define the hitting time of $A$ as
\begin{equation}
	\tau_A=\inf\{t> 0: \, \si_t \in A\}.
\end{equation}

Notice that $H_N$, $\mu_{\beta, N}$ and $p_N$ are random variables, with respect to the random realisation of the random variables $\{J_{ij}\}_{i,j\in[N]}$. In this paper the results involving these random variables hold pointwise, namely for every realisation of $\{J_{ij}\}_{i,j\in[N]}$, unless we specify it differently, as in our main theorems.

\subsection{The Curie--Weiss model}\label{sec:cwmodel}
Before stating the main results, we recall  some results for the
mean-field Curie--Weiss (CW) model  
(see e.g. Bovier and den Hollander~\cite[Section 13]{BdH} and Bovier, Eckhoff, Gayrard and Klein~\cite{BEGK1}).
The CW Hamiltonian $\tilde H_N$ can be obtained taking the mean value of \eqref{eq:hamER} (namely, the first equality in \eqref{eq:H_CW} below).
A simplifying feature of the CW model is that its Hamiltonian depends on the configuration $\si \in \Ss_N$ 
only through the empirical magnetisation $m_N: \Ss_N\to \Gamma_N$
defined as
\begin{equation} \label{eq:def_m}
	m_N(\si)=\frac{1}{N}\sum_{i=1}^N \si_i 
	\in \Gamma_N=
	\left\{-1,-1+\tfrac{2}{N}, ...,1-\tfrac{2}{N}, 1\right\}.
\end{equation}
From now on we will drop the dependency on $N$ from the magnetisation.
Then we can write
\begin{equation} \label{eq:H_CW}
	\tilde H_N(\si)=- \frac{1}{N} \sum_{1\leq i< j\leq N}  \sigma_i \sigma_j - h \sum_{i\in[N]}\sigma_i=-N\left(\tfrac{1}{2}m(\si)^2 + h m(\si)\right)
\end{equation}
and we can define, for any $m \in \GN$, 
\begin{equation}
	E(m)=-\tfrac{1}{2}m^2 - h m,
\end{equation}
obtaining 
\begin{equation}\label{eq:tiHNE}
	\tilde H_N(\si)= N E(m(\si)).
\end{equation}

The associated Gibbs measure is 
\begin{equation} \label{eq:ti_mu}
	\tilde\mu_{\beta, N}(\si)= 
	\frac{\eee^{-\beta N E(m(\si))}}{\tilde Z_{\beta,N}}, \quad \si\in\Ss_N,
\end{equation}
where $\tilde Z_{\beta,N}= \sum_{\sigma\in\Ss_N} \eee^{-\beta \tilde H_N(\si)}$ is the normalising partition function. 

We denote the law of $m(\si)$ under the Gibbs measure by
\begin{equation} \label{eq:tiQ}
	\ti\Qq_{\beta,N} = \ti\mu_{\beta,N}\circ m^{-1}.
\end{equation}

Then 
\begin{equation}\label{eq:meso}
	\ti\Qq_{\beta,N}(m)=
	\frac{\eee^{-\beta N E(m)}}{\ti Z_{\beta,N}}\sum_{\si\in\Ss_N} \mathds{1}_{m(\si)=m}
	=\frac{\eee^{-\beta N E(m)}}{ \ti Z_{\beta,N}}\binom{N}{\frac{1+m}{2}N}
	=\frac{\eee^{-\beta N f_{\beta,N}(m)}}{ \ti Z_{\beta,N}},
\end{equation}
where 
\begin{equation} \label{eq:f_betaN}
	f_{\beta,N}(m)=E(m) + \beta^{-1}I_N(m)= -\frac{m^2}{2}-hm+\beta^{-1}I_N(m)
\end{equation}
is the finite volume \emph{free energy}, while the \emph{entropy} of the system is given by 
the following combinatorial coefficient
\begin{equation}\label{eq:defIN}
	I_N(m)=-\frac{1}{N}\log \binom{N}{\frac{1+m}{2}N}
\end{equation}
and it has the following properties: as $N\to\infty$, 
\begin{equation}\label{eq:limIN}
	I_N(m)\rightarrow I(m)\equiv \frac{1-m}{2} \log \frac{1-m}{2} +\frac{1+m}{2} \log \frac{1+m}{2},
\end{equation}
more precisely, 
\begin{equation} \label{eq:def_f}
	I_N(m)-I(m)=\frac 1{2N} \ln \frac {1-m^2}4+\frac {\ln N+\ln(2\pi)}{2N} +O\left(\frac{1}{N^2}\right).
\end{equation}
As reference see for example Bovier, Eckhoff, Gayrard and Klein~\cite[(7.18)]{BEGK1}.

Notice that the previous definitions imply 
\begin{equation}\label{eq:tiMuQu}
	\tilde\mu_{\beta, N}(\si)= \ti\Qq_{\beta,N}(m(\si)) \,\eee^{NI_N(m(\si))}.
\end{equation}

We use the notation $f_{\beta}(m)=\lim_{N\to\infty}f_{\beta,N}(m)$. We refer to Bovier and den Hollander~\cite[(13.2.6)]{BdH} for more details on the following result.

\begin{lemma}\label{lem:expfN}
	For $m \in (-1,1)$,
	\begin{equation}
		\eee^{-\beta N f_{\beta,N} (m) }= \eee^{-\beta N f_{\beta} (m)} (1+o(1)) \sqrt{\frac{2}{\pi N(1-m^2)}}
	\end{equation}
	and for $m \in \{ 1,-1 \}$, $ f_{\beta,N} (m) = f_{\beta} (m)$.
\end{lemma}

\begin{remark}
	Comparing our definitions and the literature (e.g. Bovier and den Hollander~\cite[Section 13.1]{BdH}), one notices that the Gibbs measure is often defined with an additional factor $2^{-N}$, corresponding to the reference measure. More precisely, the Gibbs measure would be $\tilde\mu_{\beta, N}(\si)= 
	\frac{1}{\tilde Z_{\beta,N}}\eee^{-\beta N E(m(\si))}2^{-N}$, where the partition function would be defined by $\sum_{\sigma\in\Ss_N} \eee^{-\beta \tilde H_N(\si)}2^{-N}$. 
	We preferred to discard the $2^{-N}$ from our definitions. 
	Therefore, for consistency, our definition of $I_N$ differs  from the classical one by a factor $2^{-N}$ inside the logarithm, 
	yielding a difference of $\log(2)$ in the limit in \eqref{eq:limIN} with respect to Bovier and den Hollander~\cite[(13.1.14)]{BdH} or Bovier, Eckhoff, Gayrard and Klein~\cite[(7.17)]{BEGK1}.
\end{remark}

We consider the Glauber dynamics associated to the CW Hamiltonian 
in analogy with \eqref{eq:rates} and with transition probabilities $\ti p_N(\si,\si')$. 
A particular feature of this model is that the image process $m(t)\equiv m(\si(t))$ of the 
Markov process $\si(t)$ under the map $m$ is again a Markov process on $\GN$, with transition 
probabilities 
\begin{equation} \label{eq:rates_r}
	\ti r_N(m,m')= 
	\begin{cases}
		\exp(-\beta N[E(m') - E(m)]_+ ) \frac{(1-m)}{2} \quad & \text{if }m'=m + \frac{2}{N}, \\
		\exp(-\beta N[E(m') - E(m)]_+ ) \frac{(1+m)}{2} & \text{if } m'=m - \frac{2}{N},\\
		0 &\text{else}. 
	\end{cases} \\
\end{equation}

The equilibrium CW model displays a phase transition. 
Namely, there is a critical value of the inverse temperature $\beta_c=1$ such that, in the regime $\beta>\beta_c$, $h>0$ and small, 
the free energy $f_{\beta}(m)$ is a double-well function with
local minimisers $m_-, m_+$ and saddle point $m^*$. They are the
solutions of equation $m=\tanh(\beta(m+h))$. Since $f_{\beta}(m_-)>f_\beta(m_+)$, the phase with $m_-$
represents the metastable state, while $m_+$ represents the stable state for the system. 
Define $m_{-}(N),m^*(N),m_{+}(N)$ as the closest points in $\GN$ to $m_-, m^*,m_+$ respectively, with respect to the Euclidean distance on $\RR$.
$\{m_{-}(N), m_{+}(N)\}$ form a metastable set in the sense of Definition 8.2 of Bovier and den Hollander~\cite{BdH}. 
Let $\EE^{\text{CW}}_{m_{-}(N)}$ be the expectation with respect to the Markov process $ m(t)$ with transition probabilities  $\ti r_N$ and starting at $m_{-}(N)$. Then the following theorem holds. 
\begin{theorem}\label{teo:cw}
	For $\beta>1$ and $h>0$ small enough, as $N\to\infty$,
	\begin{eqnarray}\nonumber
		\EE^{\text{CW}}_{m_{-}(N)}[\tau_{m_{+}(N)}]&=&
		\exp\Big(\beta N \left[f_{\beta}(m^*)-f_{\beta}(m_-)\right]\Big)\\
		&&\times\frac{\pi}{1+m^*}\sqrt{\frac{1-{m^*}^2}{1-m_-^2}}\frac{N(1+o(1)) }{\beta\sqrt{f''_{\beta}(m_-)\left(-f''_{\beta}(m^*)\right)}}.
	\end{eqnarray}
\end{theorem}
As a reference see Bovier and den Hollander~\cite[Theorem 13.1]{BdH}. The difference of sign in the denominator with respect to our statement is due to the fact that their result holds for $h<0$, while ours for $h>0$.

We conclude this section by giving the explicit formula of the capacity for the CW model. The definition of \emph{capacity} 
is given in \eqref{capacrepr}, while its relation with the mean hitting time is given by the key relation \eqref{eq:EtauCap}. 
Let us denote, for any subset $U$ of $\Gamma_N$, the set of configurations with magnetisation in $U$ by
\begin{equation}
	\Ss_N[U]=\{\sigma \in \Ss_N: m(\sigma) \in U \}
\end{equation} 
and for simplicity, for any $m \in \Gamma_N$, the set of configurations with given magnetisation $m$  by $\Ss_N[m]$. Notice that $\Ss_N[m]$ has cardinality $\eee^{-NI_N(m)}$, where $I_N(m)$ is defined in \eqref{eq:defIN}.

Then, the following formula, 
\begin{equation} \label{eq:capaCW}
	\capa^{\text{CW}}(\Ss_N[m_-(N)],\Ss_N[m_+(N)])
	= \frac{1}{\tilde{Z}_{\beta,N}}\mathrm{e}^{-\beta N f_\beta(m^*)}  \frac{\sqrt{\beta \left(-f''_{\beta}(m^*)\right)}}{\pi N}\sqrt{\frac{1+m^*}{1-m^*}} (1+o(1)),
\end{equation}
follows from standard arguments (see e.g. techniques used in the proof of Bovier and den Hollander~\cite[Theorem 13.1]{BdH}).

\subsection{Main results}
For any $A,B\subset \Ss_N$ disjoint,
we define the so-called \emph{last-exit biased distribution} on $A$ for the transition
from $A$ to $B$ as
\begin{equation}\label{eq:deflebdist}
	\nu_{A,B}(\sigma) = \frac{\mu_{\beta,N}(\si) \PP_\si(\tau_B<\tau_A)}{\sum_{\si\in A} \mu_{\beta,N}(\si) \PP_\si(\tau_B<\tau_A)}, \qquad \si\in A.
\end{equation}
Since we are going to use $\nu_{A,B}$ on the sets $\Ss_N[m_{-}(N)],\Ss_N[m_{+}(N)]$ defined above, we introduce the following simplified notation
\begin{equation}
	\nu^N_{m_-,m_+}=\nu_{\Ss_N[m_-(N)], \Ss_N[m_+(N)]}.
\end{equation}

The following theorem gives a description of the dynamical properties of the RDCW model
in the \emph{metastable regime} where $h$ is positive and small 
enough, $\beta>\beta_c=1$ ($\beta_c$ is the critical inverse temperature for the RDCW model) and  $N$ is going to infinity.
We provide an estimate on the mean  time it takes to 
the system, starting with initial distribution $\nu^N_{m_-,m_+}$, to reach $\Ss_N[m_+(N)]$.
More precisely, we estimate, in the limit as $N \to \infty$, its ratio with the mean metastable exit time
for the CW model to go from $m_-(N)$ to $m_+(N)$, providing 
constant upper and lower bounds independent of $N$. 
Because of the random interaction, the result is given in the form of 
tail bounds. 

After recalling that notation $\PP_J$ and $\EE_{\nu}$ was introduced in Section~\ref{sec:GdER}, while $\EE^{\text{CW}}_{m_-(N)}$ was introduced in Section~\ref{sec:cwmodel}, we are ready to formulate our main theorem.

\begin{theorem}[Mean metastable exit time]\label{teo:exptime}
	For $\beta>1$, $h>0$ small enough and for $s>0$,
	there exist absolute constants $k_1,k_2>0$ and $C_1(p,\beta)<C_2(p,\beta,h)$ independent of $N$, 
	such that
	\begin{equation}
		\lim_{N\uparrow\infty}	\PP_J\left(	 C_1  \eee^{-s} (1+o(1))
		\leq \frac{ \EE_{\nu^N_{m_-,m_+}} \left[\tau_{\Ss_N[m_+(N)]}\right]}{\EE^{\text{CW}}_{m_-(N)}\left[\tau_{m_+(N)}\right]}
		\leq C_2  \eee^{s} (1+o(1))\right)\geq 1-k_1\eee^{-k_2 s^2}.
	\end{equation}	
	
\end{theorem}

The quantities $C_1$ and $C_2$ in the previous theorem can be explicitly written. Set
\begin{equation} \label{eq:alpha_k}
	\alpha= \frac{\beta^2 (1-p)}{4p}, \qquad
	\kappa=\alpha+ \max_{\eta\in(0,1)}
	\left\{ \log \eta -  \frac{\beta \sqrt{2\alpha +  \log\left(\frac{c_1}{(1-\eta)^2}\right)} }{p\sqrt{2c_2}} \right\},
\end{equation}
where $c_1, c_2>0$ are absolute constants coming from Theorem~\ref{thm:TalConcIneq}. It is easy to see that
$\kappa<\alpha$. With this notation
\begin{equation}
	C_1 = C_1(\beta, h, p)= \eee^{-2\beta(1+h) -\alpha + \kappa}, 
\end{equation}
\begin{equation}
	C_2 = C_2(\beta, h, p) =\eee^{2\beta(1+h) +2\alpha}.
\end{equation}

\medskip

\subsection{Proof of the main theorem}\label{sec:ideaproof}
The proof of Theorem~\ref{teo:exptime} is based on the \emph{potential theoretic approach} to metastability, 
which turns out to be a rather powerful tool to analyse the main object we are interested in, i.e. the mean hitting time of $\Ss_N[m_+(N)]$ for the system with initial distribution $\nu^N_{m_-,m_+}$. 
The general ideas of this approach 
were first introduced in a series of papers by 
Bovier, Eckhoff, Gayrard and Klein~\cite{BEGK1,BEGK2,BEGK3}. We refer to 
Bovier and den Hollander~\cite{BdH} for an overview on this method. 

The crucial formula in the study of metastability is given by the following 
relation linking mean 
hitting time and \emph{capacity} of two sets $A,B\in \Ss_N$, which can be found in Bovier and den Hollander~\cite[Eq. (7.1.41)]{BdH}
\begin{equation} \label{eq:EtauCap}
	\EE_{\nu_{A,B}}[\tau_B]=
	\sum_{\si\in A} \nu_{A,B}(\sigma)\EE_{\si}[\tau_B]=
	\frac{1}{\capa(A,B)}\sum_{\si'\in \Ss_N} \mu_{\beta,N}(\si') h_{AB}(\si'),
\end{equation}
where the capacity, as in Bovier and den Hollander~\cite[(7.1.39)]{BdH}, is defined by
\begin{equation}
	\label{capacrepr}
	\capa(A,B)=\sum_{\si\in A}\mu_{\beta,N}(\si) \PP_{\si}(\tau_B<\tau_A).
\end{equation}
The function $h_{AB}$ is called \emph{harmonic function} and has the following 
probabilistic interpretation
\begin{equation} \label{eq:h_prop}
	h_{A B}(\si) =\left\{\begin{array}{ll}
		\PP_\si(\tau_A<\tau_B) &\qquad \si \in \Ss_N \setminus (A \cup B),\\
		\mathds{1}_{A}(\si) &\qquad \si \in A\cup B.\\
	\end{array}
	\right.
\end{equation}
We refer to Bovier and den Hollander~\cite[Section 7.1.2]{BdH} for further details on the latter quantities.

By \eqref{eq:EtauCap}, in order to estimate mean hitting times one needs estimates both on the capacity and on the harmonic function.

We prove bounds on 
the capacity of two sets $\Ss_N[m_1],\Ss_N[m_2]$, stated in the two following theorems.

\begin{theorem}\label{thm:upperB}
	For any $m_1 \neq m_2 \in \GN$ and any $s>0$, there exist absolute constants $k_1,k_2>0$ such that  
	\begin{equation}
		\PP_{J}\left(\frac{
			Z_{\beta,N}\,\capa\left(\Ss_N[m_1],\Ss_N[m_2]\right)}{\tilde{Z}_{\beta,N}\,\capa^{\text{CW}}\left(\Ss_N[m_1],\Ss_N[m_2]\right) } \leq  \eee^{s+2\beta(1+h) + \alpha}(1+o(1))  \right)\geq 1-k_1\eee^{-k_2 s^2},
	\end{equation} 
	asymptotically as $N \to \infty$, where $\alpha$ is defined in \eqref{eq:alpha_k}.
\end{theorem}
\begin{theorem}\label{thm:lowerB}
	For any $m_1\neq  m_2 \in \GN$ and any $s>0$, there exist absolute constants $k_1,k_2>0$ such that  
	\begin{equation}
		\PP_J\,\left(
		\frac{Z_{\beta,N}\,\capa\left(\Ss_N[m_1],\Ss_N[m_2]\right)}{ \tilde{Z}_{\beta,N}\, \capa^{\text{CW}}\left(\Ss_N[m_1],\Ss_N[m_2]\right)}
		\geq  \eee^{-\left(s+2\beta(1+h) + \alpha\right)}(1+o(1)) \right) \geq 1-k_1\eee^{-k_2 s^2},
	\end{equation}
	asymptotically as $N \to \infty$, where $\alpha$ is defined in \eqref{eq:alpha_k}.
\end{theorem}

We state asymptotic 
upper and lower bounds on the sum over the harmonic function in the numerator of \eqref{eq:EtauCap} in the following proposition. We used the simplified notation 
\begin{equation}
	h^N_{m_-,m_+}=h_{\Ss_N[m_-(N)],\Ss_N[m_+(N)]}.  
\end{equation}
\begin{theorem} \label{thm:upperHarm}
	For any $s>0$, there exist absolute constants $k_1,k_2>0$ such that  
	
	\begin{equation}\label{eq:upperHarm}
		\PP_J\left(	\sum_{\si \in \Ss_N} \mu_{\beta,N}(\si)   h^N_{m_-,m_+}(\si) \leq  \eee^{\alpha+s} \, \frac{\exp\big(-\beta N f_{\beta} (m_-)\big) \,\big(1+o(1)\big) } {Z_{\beta,N} \sqrt{\left(1-m_-^2\right) \, \beta  f''_{\beta} (m_-) }}\right) 
		\geq 1-k_1\eee^{-k_2s^2},
	\end{equation}
	\begin{equation}
		\PP_J\left(		\sum_{\si \in \Ss_N} \mu_{\beta,N}(\si)  h^N_{m_-,m_+}(\si) \geq
		\eee^{\kappa-s}\, \frac{\exp\big(-\beta N f_{\beta} (m_-)\big) \,\big(1+o(1)\big)} {Z_{\beta,N} \sqrt{\left(1-m_-^2\right) \, \beta  f''_{\beta} (m_-) }}  \right) 
		\geq 1-k_1\eee^{-k_2s^2},
	\end{equation}
	asymptotically as $N \to \infty$, and where $\alpha$ and $\kappa$ are defined in \eqref{eq:alpha_k}.
\end{theorem}

We conclude this section using Theorems~\ref{thm:upperB}-\ref{thm:upperHarm}, 
to prove the main theorem. First, we introduce the following notation which will be extensively 
used: 
\begin{equation}\label{eq:Psnotation}
	A\overset{P(s)}{\gtreqless}B
	\qquad \text{ is equivalent to } \qquad
	\PP_{J}(A\gtreqless B)\geq 1-k_1 \eee^{-k_2 s^2},
\end{equation}
for all $s>0$ and for some absolute constants $k_1,k_2>0$, whose values might change along the paper.

\begin{proof}[Proof of Theorem~\ref{teo:exptime} ]
	We prove here only the upper bound, as the lower bound follows similarly. 
	More precisely, we prove
	\begin{equation}
		\frac{ \EE_{\nu^N_{m_-,m_+}} \left[\tau_{\Ss_N[m_+(N)]}\right]}{\EE^{\text{CW}}_{m_-(N)}\left[\tau_{m_+(N)}\right]}
		\overset{P(s)}{\leq} C_2  \eee^{s}.
	\end{equation}
	We start from \eqref{eq:EtauCap}, which in our case reads
	\begin{equation}\label{eq:EtauCap1}
		\EE_{\nu^N_{m_-,m_+}} \left[\tau_{\Ss_N[m_+(N)]}\right]
		= \frac{\sum_{\si \in \Ss_N}\mu_{\beta,N}(\si)  h^N_{m_-,m_+}(\si) }{\text{cap}\left(\Ss_N[m_-(N)],\Ss_N[m_+(N)]\right)} .
	\end{equation}
	From \eqref{eq:upperHarm} we obtain
	\begin{equation}
		\EE_{\nu^N_{m_-,m_+}} \left[\tau_{\Ss_N[m_+(N)]}\right]
		\overset{P(s)}{\leq} \frac{\eee^{\alpha+s} \, \exp\big(-\beta N f_{\beta} (m_-)\big)(1+o(1)) }
		{Z_{\beta,N}\,\text{cap}(\Ss_N[m_-(N)],\Ss_N[m_+(N)]) \sqrt{\left(1-m_-^2\right) \beta  f''_{\beta} (m_-) }}.
	\end{equation}

	Via the lower bound on the capacity from Theorem~\ref{thm:lowerB}, we  obtain
	\begin{equation}
		\begin{split}
			&\EE_{\nu^N_{m_-,m_+}} \left[\tau_{\Ss_N[m_+(N)]}\right]\\
			& \quad \overset{P(s)}{\leq} \eee^{2s+2\beta(1+h) + 2\alpha}\,
			\sqrt{\frac{1-m^*}{1+m^*}} \,\frac{\pi \, N\, \exp\left(\beta N \left[f_{\beta}(m^*)-f_{\beta} (m_-)\right]\right)
			}{ \beta\sqrt{\left(1-m_-^2\right)   f''_{\beta} (m_-) \left(-f_{\beta}''(m^*)\right) }}(1+o(1))    \\
			&\quad =\eee^{2s+2\beta(1+h) + 2\alpha} \, \EE^{\text{CW}}_{m_{-}(N)}[\tau_{m_{+}(N)}],
		\end{split}
	\end{equation}
	where we used \eqref{eq:capaCW}	and Theorem~\ref{teo:cw}.	
\end{proof}

\subsection{Outline} The remainder of this paper is organised as follows. In Section~\ref{sec:two} we use the powerful 
\emph{Talagrand's concentration inequality} to obtain bounds on the equilibrium measure of the RDCW model.
These bounds allow us to write the RDCW \emph{mesoscopic} measure in terms of the deterministic CW one, 
times a random factor which is the exponential of a sub-Gaussian random variable.
In Section~\ref{sec:capacities} we give the proof of  Theorems~\ref{thm:upperB} and \ref{thm:lowerB}   
via two dual variational principles, the Dirichlet and the Thomson principles, which are
the building blocks of the potential theoretic approach to metastability. In obtaining upper and lower 
bounds on the
capacity, the main strategy is to use the results of Section~\ref{sec:two} in order to recover the capacity 
of the CW model. In Section~\ref{sec:harmonicSum}
we prove Theorem~\ref{thm:upperHarm}, i.e. we
compute the asymptotics of the numerator in the formula for the mean hitting time using estimates on the
harmonic function.


\section{Equilibrium analysis via Talagrand's concentration inequality}\label{sec:two}
In this section we prove that the equilibrium mesoscopic measure of the RDCW model
is in fact very close to that of the CW model. This is done in two steps. First,
we prove that the difference between the \emph{random free energy} at fixed magnetisation
and its average can be controlled via \emph{Talagrand's concentration 
	inequality}. Second, we find upper and lower bounds on the aforementioned average 
by estimating first and second moments of the partition function of the RDCW model 
at fixed magnetisation.

\subsection{Mesoscopic measure and closeness to the CW model}
We start by analysing the equilibrium measure of the RDCW model. The aim is to express the   
equilibrium measure $\mu_{\beta, N}$, defined in \eqref{eq:eqmeas},
in terms of the empirical magnetisation in order to obtain a \emph{mesoscopic} description, as we did for the CW 
model in Section~\ref{sec:cwmodel}.
Let us define the measure $\Qq_{\beta,N}$ on $\GN$, and let the partition function be its normalisation
\begin{equation}\label{eq:defQ}
	\Qq_{\beta,N}(\cdot) = \mu_{\beta,N}\circ m^{-1}(\cdot) =\sum_{\si \in \Ss_N[\cdot]} \mu_{\beta,N}(\si),
	\qquad Z_{\beta,N}=  \sum_{m\in\Gamma_N} \Qq_{\beta,N}(m).
\end{equation}
A priori the Hamiltonian of the RDCW model is not only depending on $m$, but it depends of course
on the whole spin configuration. Nonetheless, we will see later in this section that the mesoscopic measure 
$\Qq_{\beta,N}$ can be written in terms of the mesoscopic measure $\ti{\Qq}_{\beta,N}$ 
of the standard CW model.

\begin{equation}\label{eq:EH}
	\EE[H_N(\sigma)]= - \frac{1}{Np} \sum_{i<j}\mathbb{E}[J_{i j}] \sigma_i \sigma_j - h \sum_{i}\sigma_i 
	= - \frac{p}{Np} \sum_{i<j} \sigma_i \sigma_j - h \sum_{i}\sigma_i = \tilde{H}_N(\sigma).
\end{equation}
Therefore, 
we can split the Hamiltonian into the mean-field part and the remaining random part
obtaining
\begin{equation} \label{eq:Hdecomp}
	H_N(\si)=  \EE[H_N(\sigma)] + \Delta_{N,p}(\sigma),
\end{equation}
where, introducing the notation  $\hat{J}_{i j}= J_{i j} - p$, 
\begin{equation}\label{eq:delta}
	\Delta_{N,p}(\sigma) =H_N(\sigma) -\tilde{H}_N(\sigma)=
	-\frac{1}{Np} \sum_{i<j} \hat{J}_{i j}\sigma_i \sigma_j. 
\end{equation}
Note that $\Delta_{N,p}$ is a random variable with zero mean. 
In order to simplify the notation, we drop from now on the dependence on $N$ and $p$, 
from $\Delta_{N,p}$. Next, we  write 
the mesoscopic measure as
\begin{equation} \label{eq:ZbN}
	\Qq_{\beta,N}(m)= 
	\frac{1}{Z_{\beta,N}}\eee^{-\beta N E(m)} \cdot \sum_{\substack{\sigma \in \Ss_N[m]}}\eee^{-\beta  \Delta(\sigma)},
\end{equation}
where $E(m)$ is defined in  \eqref{eq:H_CW}. 

We will now focus on proving  bounds for functions of $\sum_{\substack{\sigma \in \Ss_N[m]}}\eee^{-\beta  \Delta(\sigma)}$ more general than $\Qq_{\beta,N}(m)$. 
These results will be fundamental to prove our main theorem in the following sections. 
We will come back to $\Qq_{\beta,N}$ at the end of this section, 
proving its closeness to the CW correspondent $\tilde\Qq_{\beta,N}$ as a consequence of those general results.

Let us introduce the 
following notation, where we drop the dependence on $\beta$ for simplicity
\begin{align}
	\mathcal{Z}_{N,g}&= \sum_{m \in  \GN } g(m) \sum_{\substack{\sigma \in \Ss_N[m]}}\eee^{-\beta  \Delta(\sigma)}=\exp\left(Np_{N,g}\right)
	\,\exp\left(N \left[F_{N,g}-p_{N,g}\right]\right) , \label{eq:defs1}\\
	F_{N,g}&=\frac{1}{N}\log \mathcal{Z}_{N,g}, \label{eq:defs2}\\
	p_{N,g}&=\mathbb{E}(F_{N,g}), \label{eq:defs3}
\end{align}
where $g:\GN \to [0,\infty)$ is a function which may depend on $N$.

We are interested in finding precise estimates on $\mathcal{Z}_{N,g}$ by writing it in terms of the entropic 
exponential term $\eee^{-NI_N(m)}$ times some random factor which takes into account the randomness of the
couplings. 
We notice that $\mathcal{Z}_{N,g}$ is the product of a deterministic factor $\eee^{Np_{N,g}}$ and a random factor $\eee^{N (F_{N,g}-p_{N,g})}$. 

We first  characterise the random variable $N (F_{N,g}-p_{N,g})$ in the following Proposition.

\begin{proposition}  \label{pro:Ysubg}
	For any $\beta$, $t>0$, 
	\begin{equation}\label{eq:Ysubg}
		\mathbb{P}_J\bigg(\abs{N(F_{N,g}-p_{N,g})} \geq t \bigg)\leq c_1 \exp\bigg(- \gamma t^2\bigg),
	\end{equation}
	where $\gamma \propto \frac{p^2}{\beta^2}$. 
\end{proposition}

The previous result  intuitively means that the \emph{random} 
$F_{N,g}$  is in fact very well concentrated around its mean $p_{N,g}$.

As a second step we provide asymptotic bounds on the average of $F_{N,g}$, i.e. the \emph{deterministic} term $p_{N,g}$.

\begin{lemma}\label{lem:upper_p}
	Asymptotically, as $N\to\infty$,
	\begin{equation}
		p_{N,g}  \leq \frac{\alpha}{N} +\frac{1}{N}\log \left(\sum_{m \in  \GN } g(m) \exp \big(-NI_N(m)\big)\right) + o\left(\frac{1}{N}\right),
	\end{equation}	
	where $I_N(m)$ is defined in \eqref{eq:defIN} and $\alpha$ in \eqref{eq:alpha_k}.
\end{lemma}

\begin{lemma}\label{lem:lower_p}
	Asymptotically, as $N\to\infty$,
	\begin{equation}
		p_{N,g}  \geq  \frac{\kappa}{N} +\frac{1}{N}\log \left(\sum_{m \in  \GN } g(m) \exp \big(-NI_N(m)\big)\right) + o\left(\frac{1}{N}\right),
	\end{equation} 	
	where $I_N(m)$ is defined in \eqref{eq:defIN} and $\kappa$ in \eqref{eq:alpha_k}.
\end{lemma}

Proposition~\ref{pro:Ysubg} together with Lemmas~\ref{lem:upper_p} and \ref{lem:lower_p} imply
the following result.
\begin{proposition}\label{pro:ZNg}
	Asymptotically, as $N\to\infty$, we have
	\begin{equation} \label{eq:Z_upper}
		\mathcal{Z}_{N,g}\leq  \eee^{\alpha} \left(\sum_{m \in  \GN } g(m) \exp \big(-NI_N(m)\big)\right) \exp\big[N(F_{N,g}-p_{N,g})\big]\left(1+o(1)\right),
	\end{equation}
	and
	\begin{equation} \label{eq:Z_lower}
		\mathcal{Z}_{N,g}\geq \eee^{\kappa}  \left(\sum_{m \in  \GN } g(m) \exp \big(-NI_N(m)\big)\right) \exp\big[N(F_{N,g}-p_{N,g})\big]\left(1+o(1)\right),
	\end{equation}
	where $	\mathcal{Z}_{N,g}$ is defined in \eqref{eq:defs1}, $\alpha$ and $\kappa$ in \eqref{eq:alpha_k}, and $I_N(m)$ in \eqref{eq:defIN}.
	Moreover,\break $N(F_{N,g}-p_{N,g})$ is a sub-Gaussian random variable with variance
	\begin{equation}
		\var\big[N(F_{N,g}-p_{N,g}) \big] \leq \frac{c \, \beta^2}{p^2},
	\end{equation}
	where $c$ is a positive constant.
\end{proposition}

We prove Proposition~\ref{pro:Ysubg}  in Section~\ref{sec:proof_Ysubg},  and Lemmas~\ref{lem:upper_p} and \ref{lem:lower_p} in Section~\ref{sec:proof_pNm}.

We are ready to state the main result of this section, as a corollary of Proposition~\ref{pro:Ysubg}  and Proposition~\ref{pro:ZNg}.
\begin{corollary}\label{cor:theBound}
	Asymptotically, as $N\to\infty$, using notation \eqref{eq:Psnotation}, the following bounds hold for any $\beta>0$ and any function  $g: \GN \to [0,\infty)$ 
	\begin{equation} \label{eq:upperBound}
		\sum_{m \in  \GN } g(m) \sum_{\substack{\sigma \in \Ss_N[m]}}\eee^{-\beta  \Delta(\sigma)} \overset{P(s)}{\leq}  \eee^{s+\alpha} \left(\sum_{m \in  \GN } g(m) \exp \big(-NI_N(m)\big)\right) \left(1+o(1)\right),
	\end{equation}
	\begin{equation}\label{eq:lowerBound} 
		\sum_{m \in  \GN } g(m) \sum_{\substack{\sigma \in \Ss_N[m]}}\eee^{-\beta  \Delta(\sigma)}\overset{P(s)}{\geq} \eee^{-s+\kappa}  \left(\sum_{m \in  \GN } g(m) \exp \big(-NI_N(m)\big)\right) \left(1+o(1)\right),
	\end{equation}
	where $\alpha$ and $\kappa$ are defined in \eqref{eq:alpha_k}, $I_N(m)$ in \eqref{eq:defIN} and $\Delta(\sigma)$ in \eqref{eq:delta}.
\end{corollary}
\begin{proof}
	By Proposition~\ref{pro:Ysubg} we obtain, for any fixed $s>0$, 	
	\begin{equation}\label{eq:upper_expNF-p}
		\exp\big[N (F_{N,g}-p_{N,g})\big] \overset{P(s)}{\leq}  \eee^{s}
		\quad \text{ and } \quad
		\exp\big[N (F_{N,g}-p_{N,g})\big] \overset{P(s)}{\geq}  \eee^{-s},
	\end{equation}
	where  $k_1,k_2>0$ are absolute constants.
	
	To conclude the proof it is sufficient to use the definition of $	\mathcal{Z}_{N,g}$ \eqref{eq:defs1} and Proposition~\ref{pro:ZNg}.
\end{proof}

\begin{remark}\label{rem:-beta}
	The exact same statement of Corollary~\ref{cor:theBound} holds replacing $\eee^{-\beta  \Delta(\sigma)}$ with $\eee^{\, \beta  \Delta(\sigma)}$. 
	The proof remains the same: the Lipschitz constant for the Talagrand concentration inequality (in Section~\ref{sec:proof_Ysubg}) is the same and the change of sign, being squared, disappears from \eqref{eq:proofEZ} onwards.
\end{remark}

We conclude this section with an immediate  application of Corollary~\ref{cor:theBound} which states the closeness of the random mesoscopic measure  $\Qq_{\beta,N}$ to the correspondent deterministic CW quantity $\ti{\mathcal{Q}}_{\beta,N}$. This result will be widely used in Section~\ref{sec:harmonicSum}.
\begin{corollary}\label{cor:sumQ}
	Asymptotically, as $N\to\infty$, using notation \eqref{eq:Psnotation}, the following bounds hold for any fixed $s>0$ and any function $\bar{g}:\GN \to [0,\infty)$ 
	\begin{equation}\label{eq:sumQUpp1}
		\begin{split}
			\sum_{m \in \GN} \bar{g}(m) \, \Qq_{\beta,N}(m) \overset{P(s)}{\leq} \eee^{s+\alpha}  \, \frac{\ti{Z}_{\beta,N}}{Z_{\beta,N}}\left(\sum_{m \in \GN} \bar{g}(m)  \,\ti{\mathcal{Q}}_{\beta,N}(m) \right)
			\left(1+o(1)\right),
		\end{split}
	\end{equation}
	
	\begin{equation}\label{eq:sumQUpp2}
		\begin{split}
			&\sum_{m \in \GN} \bar{g}(m) \, \Qq_{\beta,N}(m)\\ &\quad\overset{P(s)}{\leq} \eee^{s+\alpha}  \, \frac{1}{Z_{\beta,N}}\left(\sum_{m \in \GN \setminus \{1,-1\}} \bar{g}(m)  \, \exp{\Big(-\beta N f_{\beta} (m)\Big) } \sqrt{\frac{2}{\pi N (1-m^2)}} \right)
			\left(1+o(1)\right)\\
			& \qquad +\eee^{s+\alpha}  \, \frac{1}{Z_{\beta,N}}\left(\sum_{m \in \{1,-1\}} \bar{g}(m)  \,  \exp{\Big(-\beta N f_{\beta} (m)\Big) }  \right)
			\left(1+o(1)\right),\\
		\end{split}
	\end{equation}
	where	$\alpha$ and $\kappa$ are defined in \eqref{eq:alpha_k}.
\end{corollary}

\begin{proof}
	Using \eqref{eq:ZbN} we obtain
	\begin{equation}\label{eq:sumQ1}
		\begin{split}
			\sum_{m \in \GN} \bar{g}(m) \,  \Qq_{\beta,N}(m) &
			=\frac{1}{Z_{\beta,N}}\sum_{m \in \GN} \bar{g}(m) \, \eee^{-\beta N E(m)}   \sum_{\substack{\sigma \in \Ss_N[m]}}\eee^{-\beta  \Delta(\sigma)}. \\
		\end{split}
	\end{equation}
	Now we can apply the upper bound in Corollary~\ref{cor:theBound}, with $g(m)=\frac{1}{Z_{\beta,N}} \bar{g}(m) \, \eee^{-\beta N E(m)} $, to the right hand side of \eqref{eq:sumQ1}. We conclude the proof of \eqref{eq:sumQUpp1} using the definition of  $\ti{\mathcal{Q}}_{\beta,N}$ \eqref{eq:meso} and \eqref{eq:f_betaN}.
	
	\eqref{eq:sumQUpp2} follows by \eqref{eq:sumQUpp1} simply applying Lemma~\ref{lem:expfN}.
\end{proof}

\medskip

\subsection{Sub-Gaussian bounds on the random term}\label{sec:proof_Ysubg}

Proposition~\ref{pro:Ysubg} follows from 
Talagrand's concentration inequality,
which we cite for completeness in the version of Tao~\cite[Theorem 2.1.13]{Tao12}.

\begin{theorem}[Talagrand concentration inequality]\label{thm:TalConcIneq} 
	Let $G: \RR^M \to \RR$ be a $1$-Lipschitz and  convex function. 
	Let $M \in \NN$, $X=(X_1, \dots, X_M)$, with $X_i$ be independent r.v., uniformly bounded by $K>0$, i.e. $|{X_i} |\leq K$, for every $1 \leq i \leq M$.
	Then, for any  $t\geq 0$, 
	\begin{equation}
		\mathbb{P}\Big(|G(X)-\mathbb{E}\,G(X)| \geq t K\Big)\leq c_1  \exp\big(-c_2 t^2\big),
	\end{equation}
	with positive absolute constants $c_1$, $c_2$.
\end{theorem}
\begin{proof}[Proof of Proposition~\ref{pro:Ysubg}]
	We can apply Theorem~\ref{thm:TalConcIneq} to the free energies $F_{N,g}$ as a function of the $N^2$ coupling constants $\hat J_{ij}$.
	Indeed it is standard to see that $F_{N,g}$ is convex and Lipschitz continuous with constant $\frac{\beta}{N p\sqrt{2}}$ 
	(see e.g. Talagrand~\cite[Corollary 2.2.5]{Talabook} ).
	Thus, applying Theorem~\ref{thm:TalConcIneq} for $G =F_{N,g} \left(\frac{\beta}{N p\sqrt{2}}\right)^{-1}$ and $K = 1$, after defining $t'=t \frac{\beta}{N p\sqrt{2}}$ we obtain,
	for some positive constants $c_1,c_2$ and for any $t'\geq 0$,
	\begin{equation}\label{eq:Ysubg_b}
		\mathbb{P}_J\,\Big(N\vert{F_{N,g}-p_{N,g}}\vert \geq t' \Big) \leq c_1  \exp\Bigg(-c_2 \frac{2p^2} {\beta^2}t'^2\Bigg),
	\end{equation}
	concluding   the proof of \eqref{eq:Ysubg} and hence Proposition~\ref{pro:Ysubg}.
\end{proof}

\medskip

\subsection{Asymptotic bounds on the deterministic term} \label{sec:proof_pNm}
In this section we prove first the upper bound on $p_{N,g}$ (Lemma~\ref{lem:upper_p}) and then the lower bound (Lemma~\ref{lem:lower_p}).
The upper bound is obtained by estimates on the first moment of the random partition function $\mathcal{Z}_{N,g}$, while 
the lower bound is in the spirit of Talagrand~\cite[Theorem 2.2.1]{Talabook} and is more delicate. We will see that it involves
also estimates on the second moment of the random partition function.

\begin{proof}[Proof of Lemma~\ref{lem:upper_p}]
	Observing that $\{\hat{J}_{ij}\}_{i,j\in[N]}$ defined in \eqref{eq:delta} are i.i.d. random variables
	such that $\mathbb{E}\,\hat{J}_{ij}=0$, we easily obtain 
	\begin{equation}\label{eq:indepJ}
		\begin{split}
			\mathbb{E}[\mathcal{Z}_{N,g}]
			&= \sum_{m \in  \GN } g(m) \sum_{\substack{\sigma \in  \Ss_N[m]}} \mathbb{E}\left(\exp\left[ \frac{\beta}{Np}\sum_{i<j}\hat{J}_{i j} \sigma_i \sigma_j \right]\right) \\
			&=\sum_{m \in  \GN } g(m)  \sum_{\substack{\sigma \in \Ss_N[m]}} \prod_{i<j}\mathbb{E}\left(\exp\left[\frac{\beta}{Np}\hat{J}_{i j} \sigma_i \sigma_j \right]\right).
		\end{split}
	\end{equation}
	In order to find estimates for \eqref{eq:indepJ}, we first define
	\begin{equation}
		\Phi(x):=\mathbb{E}\,[\exp(x \hat{J}_{i j})],
	\end{equation}
	which is a function independent of $ i, j$, being $\{\hat{J}_{ij}\}_{i,j}$ i.i.d.,  
	with first and second derivatives 
	\begin{align}
		\Phi'(0)&=\mathbb{E}\,\hat{J}_{i j}=0,  \\
		\Phi''(0)&=\mathbb{E}\,\hat{J}^2_{i j}=p(1-p).
	\end{align}
	Performing a Taylor expansion of $\Phi$ we get	
	\begin{equation}
		\Phi(x)= \Phi(0) + x \,\Phi'(0) + \frac{x^2}{2} \Phi '' (0) +o(x^2) = 1 + \frac{x^2}{2}p(1-p) +o(x^2).
	\end{equation}
	Thus, we can exponentiate $\Phi(x)$ to obtain
	\begin{equation}
		\Phi(x)= \exp \bigg( \log \big(\Phi(x)\big)\bigg) 
		=\exp \bigg(\frac{x^2}{2} p(1-p) +o(x^2) \bigg),
	\end{equation}
	where 
	we used the  expansion  $\log(1+x)=x + o(x)$.
	Therefore, for any sequence of coefficients $x^2_{i j}$ which are independent of $i, j$
	and $\sigma$, we have the following
	\begin{equation} \label{eq:xij}
		\begin{split}
			&\sum_{m \in  \GN } g(m)  \sum_{\substack{\sigma \in \Ss_N[m]}}  \prod_{i<j}\mathbb{E}\Big[\exp(x_{i j} \hat{J}_{i j})\Big] = 
			\sum_{m \in  \GN } g(m) 	\sum_{\substack{\sigma \in \Ss_N[m]}} \prod_{i<j}\Phi(x_{i j}) \\
			& = \sum_{m \in  \GN } g(m) \sum_{\substack{\sigma \in \Ss_N[m]}} \prod_{i<j}\exp \left(\frac{x_{i j}^2}{2} p(1-p) +o(x_{i j}^2) \right)\\ 
			& =  \sum_{m \in  \GN } g(m) \, \eee^{-N I_N(m)} \exp \left(\frac{x_{i j}^2}{2} p(1-p) +o(x_{i j}^2) \right)^{N(N-1)/2}\\
			&= \sum_{m \in  \GN } g(m) \, \eee^{-N I_N(m)} \exp \left(x_{i j}^2 p(1-p)\frac{N(N-1)}{4} +o\left(x_{i j}^2 N(N-1)\right) \right),
		\end{split}
	\end{equation}
	asymptotically, for $x_{ij}\to 0$,
	where the third equality holds only if $x^2_{i j}$ is independent of $i, j$
	and $\sigma$. Moreover, we used that the cardinality of $\Ss_N[m]$ is $\eee^{-N I_N(m)}$, where $I_N(m)$ is defined in \eqref{eq:defIN}, and the cardinality of $\{(i,j)\in [N]^2:i<j\}$ is $\frac{N(N-1)}{2}$.
	
	We can apply \eqref{eq:xij} with $x_{i j} = \frac{\beta}{Np} \sigma_i \sigma_j$ because $x_{ij}^2$ is independent of $i,j$ and $\si$. Indeed $x_{ij}^2=\frac{\beta^2}{N^2p^2}$, being $\si_i,\si_j \in \{-1,+1\}$ for any $i,j \in [N]$ and $\si \in \Ss_N$. Thus, we get, asymptotically as $N\to\infty$, 
	\begin{equation}\label{eq:proofEZ}
		\begin{split}
			\mathbb{E}\left[\mathcal{Z}_{N,g} \right]
			&= \sum_{m \in  \GN } g(m) \, \eee^{-N I_N(m)} \exp \Bigg(\frac{\beta^2 (1-p)}{4p} +o(1) \bigg)  \\
			& =   \exp \Big(\alpha +o(1)\Big)  \sum_{m \in  \GN } g(m) \exp \big(-NI_N(m)\big), \\
		\end{split}
	\end{equation}
	where $\alpha$ is defined in \eqref{eq:alpha_k}.
	
	Therefore, by Jensen's inequality and \eqref{eq:proofEZ},
	we have
	\begin{equation}
		\mathbb{E} \Big[ \log \mathcal{Z}_{N,g}   \Big] \leq
		\log \Big(	\mathbb{E}[\mathcal{Z}_{N,g} ]\Big) \\
		= \alpha + o(1) +\log \left(\sum_{m \in  \GN } g(m) \exp \big(-NI_N(m)\big)\right)
	\end{equation}
	which proves  the upper bound.
\end{proof}

\begin{proof}[Proof of Lemma~\ref{lem:lower_p}]
	A key ingredient in the proof is to control the upper bound on the second moment of $\mathcal{Z}_{N,g}$, 
	i.e. prove that the following bound holds 
	\begin{equation}\label{eq:boundsec}
		\mathbb{E}\left[\mathcal{Z}^2_{N,g} \right]\leq \eee^{2\alpha} \,\mathbb{E}\left[\mathcal{Z}_{N,g} \right]^2 \left(1+o(1)\right),
	\end{equation}
	where $\alpha$ is defined in \eqref{eq:alpha_k}.
	
	We estimate $\mathbb{E}\left[\mathcal{Z}^2_{N,g} \right]$ using the first two lines of \eqref{eq:xij} with 
	\begin{equation}
		x_{ij}= \frac{\beta}{Np} \left( \sigma^{(1)}_i \sigma^{(1)}_j + \sigma^{(2)}_i \sigma^{(2)}_j  \right),
	\end{equation} 
	which hold also when $x_{ij}^2$ is not independent on $i,j$ and $\si$,

	\begin{equation}\label{eq:proofEZ2}
		\begin{split}
			&\mathbb{E}\left[\mathcal{Z}^2_{N,g} \right] 
			=\mathbb{E}\,\left[ \sum_{\substack{m,m' \in  \GN }} g(m)\,   g(m') 
			\sum_{\substack{\sigma^{(1)} \in \Ss_N[m], \\ \sigma^{(2)}\in \Ss_N[m']}} \exp \Bigg(\sum_{i <j} \frac{\beta}{Np} \hat{J}_{i j} \,\bigg( \sigma^{(1)}_i \sigma^{(1)}_j + \sigma^{(2)}_i \sigma^{(2)}_j  \bigg) \Bigg) \right] \\
			& = \sum_{\substack{m,m' \in  \GN }} g(m)\,  g(m') \, \mathbb{E}\,\left[ \sum_{\substack{\sigma^{(1)} \in \Ss_N[m], \\ \sigma^{(2)}\in \Ss_N[m']}}\exp \Bigg(\sum_{i <j} \frac{\beta}{Np} \hat{J}_{i j} \,\bigg( \sigma^{(1)}_i \sigma^{(1)}_j + \sigma^{(2)}_i \sigma^{(2)}_j  \bigg) \Bigg) \right] \\ 
			& = \sum_{\substack{m,m' \in  \GN }} g(m) \, g(m') \hspace{-6.2pt}\sum_{\substack{\sigma^{(1)} \in \Ss_N[m], \\ \sigma^{(2)}\in \Ss_N[m']}} \prod_{i<j}\exp \left(\frac{1}{2} \frac{\beta^2}{N^2p^2} \bigg( \sigma^{(1)}_i \sigma^{(1)}_j + \sigma^{(2)}_i \sigma^{(2)}_j  \bigg)^2 p(1-p) +o\left(\frac{\beta^2}{N^2}\right) \right)\\ 
			&\leq  \sum_{\substack{m,m' \in  \GN }} g(m)\,  g(m') \hspace{-6.2pt}\sum_{\substack{\sigma^{(1)} \in \Ss_N[m], \\ \sigma^{(2)}\in \Ss_N[m']}}\prod_{i<j}\exp \left(\frac{\beta^2 }{N^2p}2(1-p) +o\left(\frac{1}{N^2}\right)\right)\\
			&= \sum_{\substack{m,m' \in  \GN }} g(m) \, g(m') \,\eee^{-N I_N(m)}\,\eee^{-N I_N(m')} \exp \left(\frac{N(N-1)}{2} \left[\frac{\beta^2 }{N^2p}2(1-p) +o\left(\frac{1}{N^2}\right)\right]\right)\\
			&= \sum_{\substack{m,m' \in  \GN }} g(m) \, g(m') \, \eee^{-N I_N(m)}\eee^{-N I_N(m')} \exp \left( \beta^2 \frac{(1-p)}{p} +o(1) \right)\\
			& = \exp \left( 4 \alpha+o(1) \right) \sum_{\substack{m \in  \GN }} g(m)  \, \eee^{-N I_N(m)} \sum_{\substack{m' \in  \GN }}g(m')  \, \eee^{-N I_N(m')}  \\
			&=\eee^{2\alpha} \,\mathbb{E}\left[\mathcal{Z}_{N,g} \right]^2 \left(1+o(1)\right),
		\end{split}
	\end{equation}
	where, similarly to the last steps in \eqref{eq:xij}, we used that the cardinality of $\Ss_N[m]$ is $\eee^{-N I_N(m)}$, the cardinality of $\{(i,j)\in [N]^2:i<j\}$ is $\frac{N(N-1)}{2}$. Moreover, 
	in the last line we used \eqref{eq:proofEZ}.
	
	We recall the Paley--Zygmund inequality, which states that
	\begin{equation}\label{eq:smm}
		\mathbb{P}_J \big( X \geq \eta\,  \mathbb{E}X\big) \geq (1-\eta)^2 \frac{(\mathbb{E}X)^2}{\mathbb{E}X^2},
	\end{equation}
	for any non negative random variable $X$  and any $\eta \in (0,1)$. Using \eqref{eq:smm} with $X=\mathcal{Z}_{N,g}$, \eqref{eq:proofEZ} and \eqref{eq:proofEZ2}
	we get, asymptotically as ${N\to \infty}$,
	\begin{equation}\label{eq:firstEvent}
		\begin{split}
			&\mathbb{P}_J \left( \frac{1}{N}\log \mathcal{Z}_{N,g}  \geq  \frac{1}{N}\log \left( \eta \, \mathbb{E}\mathcal{Z}_{N,g} \right) \right) = \mathbb{P}_J \left( \frac{1}{N}\log \mathcal{Z}_{N,g}  
			\geq   \frac{1}{N}\log \left(\mathbb{E}\mathcal{Z}_{N,g} \right)  +  \frac{1}{N}\log \eta \right)\\
			&\geq  \frac{(1-\eta)^2}{\exp \Big( 2\alpha +o(1) \Big)}.  \\
		\end{split}
	\end{equation}
	Moreover, using \eqref{eq:Ysubg_b} together with \eqref{eq:defs2} and the change of variables $t'= Nt''$, we obtain $\forall\ t''>0$,  
	\begin{equation}
		\mathbb{P}_J\,\left(\bigg \lvert {\frac{1}{N}\log \mathcal{Z}_{N,g}-p_{N,g}} \bigg \rvert \geq t'' \right)\leq c_1 \exp\left(-\frac{ 2 c_2 N^2 p^2 t''^2}{\beta^2}\right).
	\end{equation}
	Thus, taking the complementary event, we get 
	\begin{equation}
		\mathbb{P}_J\,\left(-t'' \leq  {\frac{1}{N}\log \mathcal{Z}_{N,g}-p_{N,g}} \leq t'' \right)\geq 1- c_1 \exp\left(-\frac{ 2 c_2 N^2 p^2 t''^2}{\beta^2}\right).
	\end{equation}
	Now, using 
	\begin{equation}
		\mathbb{P}_J\,\left({\frac{1}{N}\log \mathcal{Z}_{N,g}-p_{N,g}} \leq t'' \right) \geq
		\mathbb{P}_J\,\left(-t'' \leq  {\frac{1}{N}\log \mathcal{Z}_{N,g}-p_{N,g}} \leq t'' \right)
	\end{equation} 
	and the change of variable $t=\frac{ N p\sqrt{2 c_2}}{\beta}  t''$ we obtain
	
	\begin{equation} \label{eq:event2}
		\mathbb{P}_J\,\left(\frac{1}{N}\log \mathcal{Z}_{N,g} \leq p_{N,g} +  \frac{t \beta}{N p \sqrt{2c_2}}\right)\geq 1- c_1 \exp(-t^2).
	\end{equation}
	Next we prove that the intersection of the events in  \eqref{eq:firstEvent} and \eqref{eq:event2} is non empty.
	\emph{Assuming}, for $\eta \in (0,1)$, that
	\begin{equation} \label{secondEvent2}
		\mathbb{P}_J\,\left(\frac{1}{N}\log \mathcal{Z}_{N,g} \leq p_{N,g} +  \frac{t \beta}{N p \sqrt{2c_2}}\right)> 1- \frac{(1-\eta)^2}{\exp \Big( 2\alpha +o(1) \Big)} 
	\end{equation}
	and comparing \eqref{eq:firstEvent} and \eqref{secondEvent2}, we notice that the sum of the probabilities of the  two events 
	\begin{equation}
		\left\{\frac{1}{N}\log \mathcal{Z}_{N,g} \leq p_{N,g} +  \frac{t \beta}{N p \sqrt{2c_2}}\right\},
	\end{equation}
	and
	\begin{equation}
		\left\{\frac{1}{N}\log \mathcal{Z}_{N,g}  \geq \frac{1}{N}\log \left(\mathbb{E}\mathcal{Z}_{N,g} \right)  +  \frac{1}{N}\log \eta \right\}
	\end{equation}
	is strictly greater than $1$. Therefore, they intersect in the not empty event 
	\begin{equation}
		\left\{\frac{1}{N}\log \left(\mathbb{E}\mathcal{Z}_{N,g} \right)  +  \frac{1}{N}\log \eta \leq \frac{1}{N}\log \mathcal{Z}_{N,g} \leq p_{N,g} +  \frac{t \beta}{N p\sqrt{2c_2}}\right\}
	\end{equation}
	which is contained in the deterministic set
	\begin{equation}
		\left\{ \frac{1}{N}\log \left(\mathbb{E}\mathcal{Z}_{N,g} \right)  +  \frac{1}{N}\log \eta \leq  p_{N,g} +  \frac{t \beta}{N p\sqrt{2c_2}}\right\}.
	\end{equation} 
	As a consequence, the latter set is non empty and, being deterministic, 
	\begin{equation}\label{eq:pN>}
		p_{N,g}  \geq \frac{1}{N}\log \left(\mathbb{E}\mathcal{Z}_{N,g} \right)  +  \frac{1}{N}\log \eta -  \frac{t \beta}{N p\sqrt{2c_2}} 
	\end{equation} 
	holds with probability 1.
	
	It remains to choose a suitable $t>0$ for  assumption \eqref{secondEvent2} to hold.  A sufficient condition is, for every $\eta \in (0,1)$,
	\begin{equation}\label{eq:cond_on_t}
		c_1 \exp(-t^2)<  \frac{(1-\eta)^2}{\exp \Big( 2\alpha +o(1) \Big)}, 
	\end{equation}
	namely 
	\begin{equation} \label{eq:real_cond_t}
		t^2 > 2\alpha + \log\left(\frac{c_1}{(1-\eta)^2}\right) +o(1).
	\end{equation} 
	Therefore, by \eqref{eq:pN>} and \eqref{eq:real_cond_t}, using \eqref{eq:proofEZ} we obtain, for every $\eta \in (0,1)$,
	\begin{equation}
		\begin{split}
			p_{N,g}&  \geq
			\frac{1}{N}\log \left(\mathbb{E}\mathcal{Z}_{N,g} \right)  +  \frac{1}{N}\log \eta -  \frac{\beta \sqrt{2\alpha +  \log\left(\frac{c_1}{(1-\eta)^2}\right) +o(1)} }{N p\sqrt{2c_2}} \\
			&=	 \frac{1}{N}\log\left(\sum_{m \in  \GN } g(m) \exp \big(-NI_N(m)\big) \right) +\frac{\kappa_{\eta}}{N} +o\left(\frac{1}{N}\right),
		\end{split}
	\end{equation} 
	where 
	\begin{equation}\label{eq:def_k}
		\kappa_{\eta}=\alpha + \log \eta -  \frac{\beta \sqrt{2\alpha +  \log\left(\frac{c_1}{(1-\eta)^2}\right)} }{p\sqrt{2c_2}}. 
	\end{equation}
	Notice that $\kappa_{\eta}<\alpha$.
	In order to obtain the best lower bound, namely the closer to the upper bound proven in Lemma~\ref{lem:upper_p}, we choose $\eta \in (0,1)$ s.t. $\alpha - \kappa_{\eta}$ is minimised and we conclude the proof. This choice motivates the maximum in the definition of $\kappa$, in \eqref{eq:alpha_k}.
\end{proof}	


\section{Capacity estimates} \label{sec:capacities}

This section is entirely devoted to 
obtain upper and lower bounds on capacities between sets with a
fixed magnetisation. These bounds are obtained via two dual 
variational principles, i.e. the \emph{Dirichlet} and \emph{Thomson principles} 
which are extensively discussed in Bovier and den Hollander~\cite[Sections 7.3.1, 7.3.2]{BdH}. The 
result will be expressed in terms of the capacity for the Curie--Weiss model, see \eqref{eq:capaCW}.
In particular, we prove Theorem~\ref{thm:upperB} in Section~\ref{sec:upperbound} and 
Theorem~\ref{thm:lowerB} in Section~\ref{sec:lowerbound}.

\subsection{Asymptotics on capacity: upper bound}\label{sec:upperbound}
In this section we prove Theorem~\ref{thm:upperB}, obtaining the upper bound 
on the capacity of the RDCW model in terms 
of the capacity of the CW model.

\begin{proof}[Proof of Theorem~\ref{thm:upperB}]	
	The main idea of the proof is to find an upper bound on the capacity via 
	the following Dirichlet principle (see Bovier and den Hollander~\cite[Section 7.3.1 and (7.1.29)]{BdH}
	for  details)
	\begin{equation}\label{eq:capadiri}
		\capa\,(\Ss_N[m_1],\Ss_N[m_2]) =
		\min_{f \in \mathcal H} \sum_{\si, \si' \in \Ss_N} \mu_{\beta,N}(\si) p_N(\si,\si') [f(\si) - f(\si ')]^2, 
	\end{equation}
	where 
	\begin{equation}\label{eq:defH}
		\mathcal H=\Big\{ f: \Ss_N \to [0,1] \text{ s.t. } f\vert_{\Ss_N[m_1]}=1, f\vert_{\Ss_N[m_2]}=0\Big\}.
	\end{equation}
	Later it will be clear that we can restrict the previous variational principle over the functions 
	on the space $\GN$, hence it is useful to define 
	\begin{equation}\label{eq:defHti}
		\ti{\mathcal H}=\Big\{ v: \GN \to [0,1]  \text{ s.t. } v(m_1)=1, v(m_2)=0  \Big\}.
	\end{equation}
	In order to simplify the notation we will often neglect the dependency on $m_1,m_2$  when this will not generate confusion.
	
	From \eqref{eq:capadiri}, in view of \eqref{eq:rates} and since $[f(\si) - f(\si')]$ vanishes for $\si=\si'$, we are left only with the terms such that $\si\sim\si'$ and obtain the following first equality in \eqref{eq:diribound}. The second equality in \eqref{eq:diribound} follows by \eqref{eq:tiHNE}, \eqref{eq:EH}, \eqref{eq:Hdecomp} and multiplying and dividing by $\exp\left(-\beta N \left[E(m(\si')) - E(m(\si))\right]_+ \right)$. The inequality in \eqref{eq:diribound} is obtained restricting the minimum on $ \mathcal{H}$ to the minimum on $\{f \in \mathcal{H}: f(\eta)=f(\eta') \, \, \forall \eta, \eta' \in \Ss_N \text{ s.t. } m(\eta)=m(\eta')\}$ and noticing that the latter is in bijection with $\ti{\mathcal H}$.
	\begin{equation}\label{eq:diribound}
		\begin{split}
			&Z_{\beta,N}\,\capa\left(\Ss_N[m_1],\Ss_N[m_2]\right) \\
			& =\min_{f \in \mathcal H}\frac{1}{N} \sum_{\si,\si' \in \Ss_N}\mathds{1}_{\si \sim \si'}  \exp\big(-\beta H_N(\si)\big)\exp\left(-\beta \left[H_N(\si') - H_N(\si )\right]_+\right) [f(\si) - f(\si ')]^2 \\
			&  = \min_{f \in \mathcal{H}}\tilde{Z}_{\beta,N}
			\sum_{m,m' \in \GN}\sum_{\substack{\si \in \Ss_N[m], \\ \si' \in \Ss_N[m'] }}  \mathds{1}_{\si \sim \si'}
			\frac{ \exp\left(-\beta NE(m(\si))\right)}{ \tilde{Z}_{\beta,N}N} \exp\left(-\beta N \left[E(m(\si')) - E(m(\si))\right]_+ \right) \\
			&\qquad \qquad \qquad \quad \times [f(\si) - f(\si')]^2 \exp\left(-\beta\Delta(\si)\right) \frac{\exp\left(-\beta  \left[H_N(\sigma')-H_N(\sigma)\right]_+\right)}{\exp\left(-\beta N \left[E(m(\si')) - E(m(\si))\right]_+ \right)}  \\
			&  \leq 	\min_{v \in \tilde{\mathcal{H}}}\tilde{Z}_{\beta,N} \sum_{m,m'\in \GN} \frac{ \exp\big(-\beta NE(m)\big)}{\tilde{Z}_{\beta,N}N}\exp\big(-\beta N  \left[E(m') -E(m)\right]_+ \big) [v(m) - v(m')]^2 \\ 
			&  \qquad \qquad \qquad\quad \times \sum_{\si \in \Ss_N[m]}  \exp\big(-\beta\Delta(\si)\big)  \sum_{\si' \in \Ss_N[m']}  \mathds{1}_{\si \sim \si'}
			\frac{\exp\big(-\beta  \left[H_N(\sigma')-H_N(\sigma)\right]_+\big)}{\exp\big(-\beta N  \left[E(m') - E(m)\right]_+ \big)}.  \\
		\end{split}
	\end{equation}
	
	We turn now to the last sum in \eqref{eq:diribound} and call this quantity $G(\si, m')$. 
	If $\si \sim \si'$, then $\si$ and $\si'$ differ on a single vertex, say $\ell \in [N]$, i.e. $\forall i \in [N]\setminus \{\ell\}$, $\si_i=\si'_i$  and $\si_{\ell}=-\si'_{\ell}$. 
	Thus, setting $m=m(\si)$ and recalling \eqref{eq:delta} and \eqref{eq:H_CW}, we can write 	
	\begin{align} 
		&\Delta(\si')-\Delta(\si)= - \frac{2}{Np} \sum_{i: i\neq \ell} \hat{J}_{i \ell} \si_i' \si_{\ell}' 
		= \frac{2}{Np} \sum_{i: i\neq \ell} \hat{J}_{i \ell} \si_i \si_{\ell},\label{eq:DeltaDiff}\\
		& \tH(\si')-\tH(\si)= \si_{\ell} \left[ \frac{2}{N} \sum_{i: i\neq \ell} \si_i  +2h\right]
		= \si_{\ell} \left[ \frac{2}{N} (Nm - \si_{\ell})  +2h\right]  \label{eq:diff_tiH}.
	\end{align}
	Moreover, using \eqref{eq:EH}, \eqref{eq:Hdecomp}, the definition of $\hat{J}_{ij}$ below \eqref{eq:Hdecomp}, the second equality in \eqref{eq:DeltaDiff} and the first equality in \eqref{eq:diff_tiH} we can write
	\begin{equation}\label{eq:diff_H}
		H_N(\sigma')-H_N(\sigma)=\tH(\si')-\tH(\si)+\Delta(\si')-\Delta(\si)= \si_{\ell} \left[ \frac{2}{Np} \sum_{i: i\neq \ell} J_{i 		\ell}\si_i  +2h\right].
	\end{equation}
	Due to the presence of the indicator function $ \mathds{1}_{\si \sim \si'}$, $G(\si, m')$  vanishes if $m' \notin \{m \pm \frac{2}{N}\}$.  Moreover, we can rewrite the sum 
	$\sum_{\si' \in \Ss_N[m']}  \mathds{1}_{\si \sim \si'}$ in terms of the single vertex $\ell \in [N]$ on which $\si$ and $\si'$ differ. Notice that if $m(\si')= m + \frac{2}{N}$ then $\si_{\ell}=-1=-\si'_{\ell}$ and if $m(\si')= m - \frac{2}{N}$ then $\si_{\ell}=1=-\si'_{\ell}$.
	
	Therefore, calling $ i^\pm (\si):=\{j \in [N] : \si_j=\pm 1\}$, and using \eqref{eq:tiHNE}, \eqref{eq:diff_tiH} and \eqref{eq:diff_H},  we obtain 
	\begin{equation}\label{eq:ratio_trans+}
		\begin{split}
			G\left(\si, m+\tfrac{2}{N}\right)=&\sum_{\ell \in i^-(\si)} 
			\frac{\exp\left(-\beta \left[- \frac{2}{Np} \sum_{i: i \neq \ell} J_{i\ell} \si_i -2h \right]_+ \right)}
			{\exp\left(-\beta \left[- \frac{2}{N} (Nm+1) -2h \right]_+ \right)} \leq N\, \frac{1-m}{2}\, \eee^{2\beta},\\
		\end{split}
	\end{equation}
	\begin{equation}\label{eq:ratio_trans-}
		\begin{split}
			G\left(\si, m-\tfrac{2}{N}\right)=&\sum_{\ell \in i^+(\si)} 
			\frac{\exp\left(-\beta \left[ \frac{2}{Np} \sum_{i: i \neq \ell} J_{i\ell} \si_i +2h \right]_+ \right)}
			{\exp\left(-\beta \left[ \frac{2}{N} (Nm-1) +2h \right]_+ \right)} \leq N\, \frac{1+m}{2} \, \eee^{2\beta(1+h)}. \\
		\end{split}
	\end{equation}
	To obtain the inequalities we used  the fact that, for any $\si$ in $\Ss_N$,  the cardinalities of $ i^- (\si)$ and $ i^+ (\si)$ are respectively $N\, \frac{1-m(\si)}{2}$ and $N\, \frac{1+m(\si)}{2}$.
	Moreover, for the inequality in \eqref{eq:ratio_trans+} we used the following elementary facts holding asymptotically in $N$,
	\begin{equation}
		\exp\left(-\beta \left[- \frac{2}{Np} \sum_{i: i \neq \ell} J_{i\ell} \si_i -2h \right]_+ \right) \leq 1,
	\end{equation}
	\begin{equation}
		\exp\left(\beta \left[- \frac{2}{N} (Nm+1) -2h \right]_+ \right) \leq \exp\left(\beta \left[- 2m - \frac{2}{N} -2h \right]_+ \right) \leq  \eee^{2\beta}.
	\end{equation}
	Similar inequalities were used to prove \eqref{eq:ratio_trans-}.
	
	Thus, using  \eqref{eq:diribound}, \eqref{eq:ratio_trans+}, \eqref{eq:ratio_trans-} we obtain 
	\begin{equation}
		\begin{split}
			Z_{\beta,N}& \,\capa\left(\Ss_N[m_1],\Ss_N[m_2]\right) \\
			&\leq \min_{v \in \tilde{\mathcal{H}}}\,\tilde{Z}_{\beta,N} \sum_{m,m' \in \GN} 
			\frac{\exp\big(-\beta NE(m)\big)}{\tilde{Z}_{\beta,N}N} \exp\left(-\beta N\left[E(m') - E(m)\right]_+ \right)  
			[v(m) - v(m')]^2\\ 
			& \qquad \times \eee^{2\beta(1+h)} \sum_{\si \in \Ss_N[m]} \exp\big(-\beta\Delta(\si)\big)   \left[N \tfrac{1+m}{2}\mathds{1}_{m-\frac{2}{N}}(m') + N \tfrac{1-m}{2}\mathds{1}_{m+\frac{2}{N}}(m') 
			\right].\\
		\end{split}
	\end{equation}
	Using the upper bound in Corollary~\ref{cor:theBound} with 
	\begin{equation}
		\begin{split}
			g(m)=&
			\sum_{m' \in \GN} 
			\frac{\exp\big(-\beta NE(m)\big)}{\tilde{Z}_{\beta,N}N} \exp\left(-\beta N\left[E(m') - E(m)\right]_+ \right)  
			[v(m) - v(m')]^2\\ 
			& \qquad \times \eee^{2\beta(1+h)} \left[N \tfrac{1+m}{2}\mathds{1}_{m-\frac{2}{N}}(m') + N \tfrac{1-m}{2}\mathds{1}_{m+\frac{2}{N}}(m') 
			\right]\\
		\end{split}
	\end{equation}
	we obtain 
	\begin{equation}\label{capU1}
		\begin{split}
			&Z_{\beta,N}\,\capa\left(\Ss_N[m_1],\Ss_N[m_2]\right)   \\ 
			& \overset{P(s)}{\leq}  \eee^{s+2\beta(1+h) + \alpha } \tilde{Z}_{\beta,N}\min_{v \in \tilde{\mathcal{H}}}\sum_{m,m' \in \GN}  \frac{\exp\left(-\beta NE(m)-N I_N(m)\right)}{\tilde{Z}_{\beta,N}N}\,\exp(-\beta N\left[E(m') - E(m)\right]_+ )  \\ 
			& \qquad \times [v(m) - v(m')]^2\,  \left[N \frac{1+m}{2}\mathds{1}_{m-\frac{2}{N}}(m') + N \frac{1-m}{2}\mathds{1}_{m+\frac{2}{N}}(m') \right]
			\left(1+o(1)\right)\\
			& = \eee^{s+2\beta(1+h) + \alpha } \tilde{Z}_{\beta,N}\min_{v \in \tilde{\mathcal{H}}}\sum_{m,m' \in \GN}   
			\tilde{\mathcal Q}(m)\,\tilde{r}(m, m')\, [v(m) - v(m')]^2\,  
			\left(1+o(1)\right)\\
			& = \eee^{s+2\beta(1+h) + \alpha} \,\tilde{Z}_{\beta,N} \,\capa^{CW} (\Ss_N[m_1],\Ss_N[m_2])\left(1+o(1)\right),
		\end{split}
	\end{equation}
	where we used notation \eqref{eq:Psnotation} and in the middle step we used \eqref{eq:meso}, the first equality in \eqref{eq:f_betaN} and \eqref{eq:rates_r}. Furthermore, we noticed that the variational form appearing 
	in the previous formula is the  Dirichlet principle (see Bovier and den Hollander~\cite[(7.1.29), (7.3.1)]{BdH}) applied to the random walk performed by the projection of the CW model dynamics onto the magnetisation space. See Section~\ref{sec:cwmodel} for the CW model. 
	
	\label{lumping} We conclude that the minimum equals the capacity of the CW model using lumping techniques.  More precisely, here we used Bovier and den Hollander~\cite[(9.3.6)]{BdH}, stating that the capacity for the  dynamics projected onto the magnetisation space equals the capacity for the CW dynamics on the configuration space,  which holds because of the CW model mean-field property. For reference on lumping see Bovier and den Hollander~\cite[Section 9.3]{BdH}.
\end{proof}

\subsection{Asymptotics on capacity: lower bound}\label{sec:lowerbound}
In this section we prove Theorem~\ref{thm:lowerB}, 
obtaining the lower bound on the capacity of the RDCW model in terms 
of the capacity of the CW model. We will prove it without loss of generality, only for $m_1 < m_2\in \GN$, because the capacity can be proven to be symmetric, using the reversibility of the dynamics.

The main idea of the proof is to find a lower bound on the capacity of the RDCW model via the Thomson principle (see e.g. Bovier and den Hollander~\cite[Theorem 7.37]{BdH}.
For $\capa\left(\Ss_N[m_1],\Ss_N[m_2]\right)$ it reads
\begin{equation}\label{eq:Thoms}
	\capa\left(\Ss_N[m_1],\Ss_N[m_2]\right) =  \sup \left\{ \frac{1}{\Dd(\bar{\Psi})}: \bar{\Psi}\in \mathcal U_{\Ss_N[m_1],\Ss_N[m_2]}\right\},
\end{equation}
where we denote by $\mathcal U_{\Ss_N[m_1],\Ss_N[m_2]}$ the space of all unitary antisymmetric flows from\break $\Ss_N[m_1]$ to $\Ss_N[m_2]$ and $\Dd$ is defined by 
\begin{equation}
	\Dd(\psi)= \frac{1}{2}\sum_{\sigma, \sigma' \in \Ss_N}\mathds{1}_{\sigma' \sim  \sigma}\, \frac{\psi (\sigma,\sigma')^2}{\mu_{\beta,N}(\sigma)\, p_N(\sigma,\sigma')} 
\end{equation}
for any $\psi:\Ss_N^2\to \RR$ antisymmetric flow.
Thus, in order to find a lower bound in terms of the capacity of the CW model we have to find a unitary flow from which we could reconstruct the CW capacity term.

For all $\sigma, \sigma' \in \Ss_N$, we define the candidate flow $\Psi_N$
as follows
\begin{equation} \label{eq:flowLower}
	\Psi_N(\sigma, \sigma')=\phi_N(m(\sigma),m(\sigma')),
\end{equation}
where, for all  $m,m' \in \GN$,   
\begin{equation} \label{eq:defPhi}
	\begin{split}
		\phi_N (m,m')=
		\begin{cases}
			\bigg[\frac{(1-m)N}{2} \exp\left(-N I_N(m)\right)  \bigg]^{-1} \quad  &\text{ if } m_1 \leq m \leq  m_2 -\frac{2}{N},  m' = m + \frac{2}{N} \\
			-\bigg[\frac{(1+m)N}{2} \exp\left(-N I_N(m)\right)  \bigg]^{-1} \quad  &\text{ if } m_1 + \frac{2}{N} \leq m \leq  m_2,  m' = m - \frac{2}{N} \\
			0  &\text{ otherwise.}
		\end{cases}
	\end{split}
\end{equation}

The proof of Theorem~\ref{thm:lowerB} is postponed after two technical intermediate 
results which are essential for it. 
The following lemma allows us to use $\Psi_N$ in the Thomson principle.
\begin{lemma}\label{lem:unitFlow}
	Let $m_1<m_2 \in \GN$. The  flow $\Psi_N$  on $\Ss_N$, defined in \eqref{eq:flowLower} is a unitary antisymmetric flow from $\Ss_N[m_1]$ to $\Ss_N[m_2]$,
	i.e. $\Psi_N \in \mathcal{U}_{\Ss_N[m_1],\Ss_N[m_2]}$.
\end{lemma}
\begin{proof}
	$\Psi_N$ is \emph{antisymmetric} because for all $m \in \GN$, i.e. 
	\begin{equation}\label{eq:binComp}
		\frac{(1+m)}{2}N \exp\big(-N I_N(m)\big) =  \frac{\left(1-\left(m-\tfrac{2}{N}\right)\right)}{2}N \exp\left(-N I_N\left(m-\tfrac{2}{N}\right)\right).
	\end{equation}
	Indeed, using  \eqref{eq:defIN}, the right hand side of \eqref{eq:binComp} writes
	\begin{equation}
		\begin{split}
			& \frac{\left(1-\left(m-\tfrac{2}{N}\right)\right)}{2}N \exp\big(-N I_N\left(m-\tfrac{2}{N}\right)\big)= \tfrac{\left(1-\left({m}-\frac{2}{N}\right)\right)N}{2}
			\binom{N}{\frac{1-\left({m}-\frac{2}{N}\right)}{2}N} \\
			&=  \tfrac{\left(1-\left({m}-\frac{2}{N}\right)\right)N}{2} \frac{N!}{\left[\tfrac{\left(1-{m}\right)N}{2}+1\right]!
				\left[\tfrac{\left(1+{m}\right)N}{2}-1\right]!}
			=\frac{N!}{\left[\tfrac{\left(1-{m}\right)N}{2}\right]!
				\left[\tfrac{\left(1+{m}\right)N}{2}-1\right]!} \\
			&= \frac{(1+m)}{2}N \, \binom{N}{\frac{1+{m}}{2}N} = \frac{(1+m)}{2}N \exp\big(-N I_N(m)\big). 
		\end{split}
	\end{equation}
	
	Next we prove that the \emph{Kirchhoff law} holds, i.e., for all ${\sigma} \in  \Ss_N \setminus (\Ss_N[m_1]\cup \Ss_N[m_2])$
	\begin{equation}\label{eq:Kl}
		\sum_{\sigma' \in \Ss_N: \sigma \sim {\sigma'}  } \Psi_N(\sigma, {\sigma'})=0.
	\end{equation}
	
	For all ${\sigma} \in  \Ss_N$ such that ${m(\si)} \notin (m_1,m_2)$, \eqref{eq:Kl} holds trivially being all terms zero, by \eqref{eq:defPhi}. 
	Now, for all ${\sigma} \in  \Ss_N$ such that ${m(\si)} \in (m_1,m_2)$, 
	\begin{equation}
		\begin{split}
			\sum_{\sigma' \in \Ss_N: \sigma \sim {\sigma'}  } \Psi_N(\sigma, {\sigma'})&= 
			\sum_{\substack{\sigma' \in \Ss_N: \sigma \sim {\sigma'},\\ m(\sigma')={m(\si)}+\frac{2}{N} }} \phi_N\left(m(\si),m(\sigma')\right)	
			+ 	\sum_{\substack{\sigma' \in \Ss_N: \sigma \sim {\sigma'},\\ m(\sigma')={m(\si)}-\frac{2}{N} }} \phi_N\left(m(\sigma), {m(\si')}\right)\\
			&=\frac{(1-m(\si))N}{2}	\bigg[\frac{(1-m(\si))N}{2} \exp\big(-N I_N(m(\si))\big)  \bigg]^{-1} 	\\
			& \quad -\frac{(1+m(\si))N}{2}\bigg[\frac{(1+m(\si))N}{2} \exp\big(-N I_N(m(\si))\big)  \bigg]^{-1} \\
			&=0,
		\end{split}
	\end{equation}
	where $\frac{(1\mp {m(\si)})N}{2}$ in the second equality are the cardinalities of the set over which we were summing, namely the number of negative, respectively positive, spins in a configuration $\si \in \Ss_N$.

	We are left to show that $\Psi_N$ is \emph{unitary}  from $\Ss_N[m_1]$ to $\Ss_N[m_2]$, namely
	\begin{equation}\label{eq:unit}
		\sum_{a \in \Ss_N[m_1]} \sum_{\sigma' \in \Ss_N : a \sim \sigma' } \Psi_N(a,\sigma') = 1 =	\sum_{b \in \Ss_N[m_2]} \sum_{\sigma \in \Ss_N : \sigma \sim b} \Psi_N(\sigma,b).
	\end{equation} 
	The left hand side of \eqref{eq:unit} equals 
	\begin{equation}
		\begin{split}
			&	\sum_{a \in \Ss_N[m_1]} \sum_{\sigma' \in \Ss_N : a \sim \sigma' } \phi_N\left(m(a), m(\sigma')\right) \\
			&= \sum_{a \in \Ss_N[m_1]} \sum_{\substack{\sigma' \in \Ss_N: a \sim \sigma', \\ 		m(\sigma')=m_1+\frac{2}{N} }} 
			\left[\tfrac{(1-m_1)N}{2} \exp\big(-N I_N(m_1)\big) \right]^{-1}  \\ 
			&=   \exp\big(-N I_N(m_1)\big)  \tfrac{(1-m_1)N}{2} \left[\tfrac{(1-m_1)N}{2} \exp\big(-N I_N(m_1)\big) \right]^{-1} =1. 
		\end{split}
	\end{equation} 
	
	The right hand side of \eqref{eq:unit} equals 
	\begin{equation}
		\begin{split}
			&	\sum_{b \in \Ss_N[m_2]} \sum_{\sigma \in \Ss_N : \sigma \sim b } \phi_N\left(m(\sigma), m(b)\right) \\
			&= \sum_{b \in \Ss_N[m_2]} \sum_{\substack{\sigma \in \Ss_N: \sigma \sim b, \\ 		m(\sigma)=m_2-\frac{2}{N} }} 
			\left[\tfrac{\left(1-\left(m_2-\tfrac{2}{N}\right)\right)}{2}N\exp \left(-N I_N \left(m_2-\frac{2}{N}\right) \right) \right]^{-1}  \\ 
			&=  \exp\left(-N I_N(m_2)\right)  \tfrac{\left(1+\left(m_2\right)\right)}{2}N  \left[\tfrac{\left(1-\left(m_2-\tfrac{2}{N}\right)\right)}{2}N\exp \left(-N I_N \left(m_2-\frac{2}{N}\right) \right) \right]^{-1}.
		\end{split}
	\end{equation}
	We use \eqref{eq:binComp} to  conclude the proof.
\end{proof}

\begin{lemma}\label{lem:claim1}
	For all $\si \in \Ss_N$ and $m' \in \GN$, the following holds
	\begin{multline}
		\sum_{\sigma' \in \Ss_N[m']}\mathds{1}_{\sigma' \sim  \sigma} \frac{\exp\left(-\beta\left[\tH(\sigma')-\tH(\sigma) \right]_+\right)}{\exp\Big(-\beta\left[H_N(\sigma')-H_N(\sigma)\right]_+\Big)} \\
		\leq \eee^{2\beta(1+h)} \,\left[N \frac{1+m(\sigma)}{2}\mathds{1}_{m(\sigma)-\frac{2}{N}}(m') + N \frac{1-m(\sigma)}{2}\mathds{1}_{m(\sigma)+\frac{2}{N}}(m') \right].
	\end{multline}
\end{lemma}
\begin{proof}
	Let $m=m(\sigma)$.
	The left hand side is non-zero only if $m' \in \left\{m+\frac{2}{N}, m-\frac{2}{N}\right\}$.\\
	Recalling the definition $ i^\pm (\si)=\{j \in [N] : \si_j=\pm 1\}$, if $m'=m+\frac{2}{N}$, we have
	\begin{equation}
		\begin{split}
			&\sum_{\sigma' \in \Ss_N\left[m+\frac{2}{N}\right]}\mathds{1}_{\sigma' \sim  \sigma} \frac{\exp\left(-\beta\left[\tH(\sigma')-\tH(\sigma) \right]_+\right)}{\exp\Big(-\beta\left[H_N(\sigma')-H_N(\sigma)\right]_+\Big)} \\
			&=  \sum_{\ell \in i^-(\sigma)} 
			\frac{\exp\left(-\beta \left[- \frac{2p}{N} (Nm+1) -2h \right]_+ \right)}
			{\exp\left(-\beta \left[- \frac{2}{N} \sum_{i: i \neq \ell} J_{i\ell} \sigma_i -2h \right]_+ \right)}  \\
			& \leq \sum_{\ell \in i^-(\sigma)} 
			\exp\left(\beta \left[- \frac{2}{N} \sum_{i: i \neq \ell} J_{i\ell} \sigma_i -2h \right]_+ \right)  
			\leq  \sum_{\ell \in i^-(\sigma)} \eee^{2\beta} =  N \frac{1-m}{2} \eee^{2\beta},
		\end{split}
	\end{equation}
	where we have used that, since $h>0$, 
	\begin{equation}
		\begin{split}
			&- \frac{2}{N} \sum_{i: i \neq \ell} J_{i\ell} \sigma_i -2h 
			\leq  \frac{2}{N} \sum_{i: i \neq \ell} |J_{i\ell} \sigma_i  |
			\leq\frac{2(N-1)}{N} \leq 2.
		\end{split}
	\end{equation}
	Similarly, if $m'=m-\frac{2}{N}$, we get
	\begin{multline}
		\sum_{\sigma' \in \Ss_N\big[m-\frac{2}{N}\big]}\mathds{1}_{\sigma' \sim  \sigma} \frac{\exp\left(-\beta\left[\tH(\sigma')-\tH(\sigma) \right]_+\right)}{\exp\Big(-\beta\left[H_N(\sigma')-H_N(\sigma)\right]_+\Big)} \\
		=  \sum_{\ell \in i^+(\sigma)} 
		\frac{\exp\left(-\beta \left[ \frac{2p}{N} (Nm-1) +2h \right]_+ \right)}
		{\exp\left(-\beta \left[ \frac{2}{N} \sum_{i: i \neq \ell} J_{i\ell} \sigma_i +2h \right]_+ \right)}  
		\leq N \frac{1+m}{2} \eee^{2\beta(1+h)}.
	\end{multline}
\end{proof}

With the previous lemmas at hand, we are now ready to prove the lower bound on the capacity. 
\begin{proof}[Proof of Theorem~\ref{thm:lowerB}]
	As we mentioned above, since the capacity is symmetric, we will prove the result only for $m_1<m_2 \in \GN$.
	
	Let $\Psi_N $ be the test flow defined in \eqref{eq:flowLower}, which by Lemma~\ref{lem:unitFlow} is in $ \mathcal U_{\Ss_N[m_1],\Ss_N[m_2]}$. Thus, using the Thomson principle \eqref{eq:Thoms}, we obtain the following bound
	\begin{equation}\label{eq:step1lower}
		\capa\left(\Ss_N[m_1],\Ss_N[m_2]\right) \geq \frac{1}{\Dd(\Psi_N )}.
	\end{equation}
	Therefore, we are interested in upper bounds on $\Dd(\Psi_N )$ which, using \eqref{eq:eqmeas}, \eqref{eq:rates} and \eqref{eq:delta}, can be written as follows
	\begin{equation} \label{eq:psi2_1}
		\begin{split}
			\Dd(\Psi_N ) =\frac{1}{2}N\sum_{\sigma, \sigma'  \in \Ss_N}\mathds{1}_{\sigma' \sim  \sigma} \,
			\frac{\phi_N (m(\sigma),m(\sigma'))^2}{\exp\left(-\beta (\tH(\sigma)+ \Delta(\sigma))\right)}
			\frac{Z_{\beta,N}}{\exp\Big(-\beta[H_N(\sigma')-H_N(\sigma)]_+\Big)}.
		\end{split}
	\end{equation}
	
	By multiplying and dividing by $\exp(-\beta[\tH(\sigma')-\tH(\sigma) ]_+)\,\tilde{Z}_{\beta,N}$, and using \eqref{eq:tiHNE}, \eqref{eq:ti_mu} and \eqref{eq:tiMuQu},
	we get
	\begin{equation}
		\begin{split}
			\Dd(\Psi_N ) &
			=  N\,\frac{ Z_{\beta,N} }{2\tilde{Z}_{\beta,N}}\sum_{m,m' \in \Gamma_N} 
			\frac{\phi_N (m,m')^2}{ \ti\Qq_{\beta,N}(m) \exp{\big(NI_N(m)\big)} \exp\left(-\beta N\left[E(m')-E(m)\right]_+\right)} \\ 
			&\qquad \times \sum_{\sigma \in \Ss_N[ m]} \exp(\beta \Delta(\sigma)) \sum_{\sigma' \in \Ss_N[m']}\mathds{1}_{\sigma' \sim  \sigma} \,\frac{\exp\left(-\beta \left[\tH(\sigma')-\tH(\sigma) \right]_+\right)}
			{\exp\Big(-\beta\left[H_N(\sigma')-H_N(\sigma)\right]_+\Big)}\\
			&
			\leq  N\,\frac{ Z_{\beta,N} }{2\tilde{Z}_{\beta,N}} \eee^{2\beta(1+h) }\sum_{m,m' \in \Gamma_N} 
			\frac{\phi_N (m,m')^2}{ \ti\Qq_{\beta,N}(m) \exp{\big(NI_N(m)\big)} \exp\left(-\beta N\left[E(m')-E(m)\right]_+\right)} \\ 
			&\qquad \times \Big[N \frac{1+m}{2}\mathds{1}_{m-\frac{2}{N}}(m') + N \frac{1-m}{2}\mathds{1}_{m+\frac{2}{N}}(m') \Big] \sum_{\sigma \in \Ss_N[ m]} \exp(\beta \Delta(\sigma)) ,\\
		\end{split}
	\end{equation}
	where we used Lemma~\ref{lem:claim1} to bound the sum over $\sigma'$, uniformly in $\si \in \Ss_N[ m]$. 
	Then, to bound the remaining sum over $\sigma$, we use the upper bound in Corollary~\ref{cor:theBound} (in the version with $\eee^{\,\beta \Delta(\sigma)}$, motivated by the Remark therein) with 
	\begin{equation}
		\begin{split}
			g(m)&=\sum_{m' \in \Gamma_N} 
			\frac{\phi_N (m,m')^2}{ \ti\Qq_{\beta,N}(m) \exp{\big(NI_N(m)\big)} \exp\left(-\beta N\left[E(m')-E(m)\right]_+\right)} \\ 
			&\qquad \times \Big[N \frac{1+m}{2}\mathds{1}_{m-\frac{2}{N}}(m') + N \frac{1-m}{2}\mathds{1}_{m+\frac{2}{N}}(m') \Big], 
		\end{split}
	\end{equation}
	obtaining, with notation \eqref{eq:Psnotation}, 
	\begin{equation}\label{eq:Dpsi}
		\begin{split}
			\Dd(\Psi_N ) 
			&\overset{P(s)}{\leq}N\,\frac{Z_{\beta,N} }{2\tilde{Z}_{\beta,N}}\sum_{m,m' \in \Gamma_N} 
			\frac{\eee^{s+2\beta(1+h) + \alpha} \,\phi_N (m,m')^2   \exp \big(-NI_N(m)\big)}{ \ti\Qq_{\beta,N}(m) \exp{\big(NI_N(m)\big)} \exp\left(-\beta N\left[E(m')-E(m)\right]_+\right)} \,  \\
			&\qquad  \times \Big[N \frac{1+m}{2}\mathds{1}_{m-\frac{2}{N}}(m') + N \frac{1-m}{2}\mathds{1}_{m+\frac{2}{N}}(m') \Big]
			\left(1+o(1)\right) \\
			&= \frac{Z_{\beta,N} }{2\tilde{Z}_{\beta,N}}	\eee^{s+2\beta(1+h) +\alpha} \sum_{m,m' \in \Gamma_N}  
			\frac{\phi_N (m,m')^2 \exp\big(-2N I_N(m)\big)}{ \tilde{\mathcal Q}_{\beta,N}(m) \,\tilde{r}_N(m,m')}\,  \\
			&\qquad  \times \Big[N \frac{1+m}{2}\mathds{1}_{m-\frac{2}{N}}(m') + N \frac{1-m}{2}\mathds{1}_{m+\frac{2}{N}}(m') \Big]^2
			\left(1+o(1)\right), \\
		\end{split}
	\end{equation}
	where in the equality we only used \eqref{eq:rates_r}.
	
	Now we first substitute $\phi_N$ defined in \eqref{eq:defPhi}  into \eqref{eq:Dpsi} and then use reversibility to obtain 
	\begin{equation}\label{eq:lastPsi}
		\begin{split}
			\Dd(\Psi_N ) 
			&\overset{P(s)}{\leq}
			\frac{Z_{\beta,N} }{2\tilde{Z}_{\beta,N}} \,\eee^{s+2\beta(1+h) + \alpha}   \sum_{\substack{m_1\leq m < m_2,\\ m \in \GN}} \frac{1}{ \tilde{\mathcal{Q}}_{\beta,N}(m)\,\tilde{r}_N\left(m,m+\frac{2}{N}\right)} \left(1+o(1)\right) \\
			&+\frac{Z_{\beta,N} }{2\tilde{Z}_{\beta,N}} \,\eee^{s+2\beta(1+h) + \alpha}   \sum_{\substack{m_1< m \leq  m_2,\\ m \in \GN}}\frac{1}{ \tilde{\mathcal{Q}}_{\beta,N}(m-\frac{2}{N})\,\tilde{r}_N\left(m-\frac{2}{N},m\right)} \left(1+o(1)\right) \\
			& = \frac{Z_{\beta,N} }{\,\tilde{Z}_{\beta,N}} \eee^{s+2\beta(1+h) + \alpha}   
			\sum_{m_1\leq m < m_2} 
			\frac{1}{\tilde{\mathcal{Q}}_{\beta,N}(m)\,\tilde{r}_N\left(m,m+\frac{2}{N}\right)}\left(1+o(1)\right), \\
		\end{split}
	\end{equation}
	where the last equality follows noticing that the two sums in the previous step are equal.
	
	Therefore, by \eqref{eq:step1lower} and \eqref{eq:lastPsi}, we obtain
	\begin{equation}\label{eq:finlowerb}
		\begin{split}
			Z_{\beta,N}\,&\capa(\Ss_N[m_1],\Ss_N[m_2]) 	\geq \frac{Z_{\beta,N}}{\Dd(\Psi_N)} \\
			&\overset{P(s)}{\geq} \tilde{Z}_{\beta,N}   \, \eee^{-s-2\beta(1+h) - \alpha}   \left[ \sum_{m_1\leq m < m_2}  \frac{1}{ \tilde{\mathcal{Q}}_{\beta,N}(m)\,\tilde{r}_N\left(m,m+\frac{2}{N}\right)} \right]^{-1}\left(1+o(1)\right) \\
			& = \tilde{Z}_{\beta,N}\, \eee^{-s-2\beta(1+h) - \alpha}  \, \capa^{\text{CW}}(\Ss_N[m_1],\Ss_N[m_2])\left(1+o(1)\right),
		\end{split}
	\end{equation}
	where  we used notation \eqref{eq:Psnotation} and we noticed that the inverse of the expression appearing 
	in brackets in \eqref{eq:finlowerb} gives exactly the capacity for the CW model. 
	Indeed, 
	that expression gives exactly the capacity for the one-dimensional random walk in $\GN$ which is the projection of the CW dynamics onto the magnetisation space $\GN$ 
	(see the formula for the capacity in Bovier and den Hollander~\cite[Section 7.1.4, (7.1.60)]{BdH}). 
	Using lumping techniques exactly as at the end of the proof of Theorem~\ref{thm:upperB} (end of Section~\ref{lumping}), we have that the aforementioned capacity equals the CW capacity.
\end{proof}

\medskip


\section{Estimates on the harmonic function} \label{sec:harmonicSum}
As pointed out in Section~\ref{sec:ideaproof}, the proof of Theorem~\ref{teo:exptime} relies on  sharp estimates on capacities, carried out in Section~\ref{sec:capacities},  and estimates on the harmonic function. We entirely devote this section to obtain asymptotic 
upper and lower bounds on the numerator in \eqref{eq:EtauCap1}, which is given by the following sum 
\begin{equation}\label{eq:harmonicsum}
	\sum_{\si \in \Ss_N} \mu_{\beta,N}(\si)   h^N_{m_-,m_+}(\si),
\end{equation}
that is to give the proof of
Theorem~\ref{thm:upperHarm}.

In order to control the sum \eqref{eq:harmonicsum}, one 
generally uses a renewal argument which relies again on estimates over capacities. However, in our case
this is not possible, due to the fact that capacities of single spins are too small.

We first  prove 
the upper bound and then give some details about how to prove the lower bound, which is very similar 
and more straightforward. Our proof follows Bianchi, Bovier and Ioffe~\cite[Section 6]{BBI09}.

\subsection{Notation and decomposition of the space} \label{sec:notat_delta_theta}
Before starting with the proof, we  introduce some notation. We refer to Figure~\ref{fig:wells} below
for a better visual understanding of the objects we are defining.

Recall that we denote  by 
$m_+$ the global minimum, by $m_-$ the local minimum,
and by $m^*$ the local maximum of $ f_{\beta}(\cdot)$ in $[-1,1]$, where $ f_{\beta}(\cdot)= \lim_{N \to \infty} f_{\beta,N}(\cdot)$, 
defined in \eqref{eq:f_betaN}.
We want to decompose the space $\GN$ (and eventually the set of spin configurations $\Ss_N$) 
according to the values of $f_{\beta}$. The notation and the decomposition are organised in 4 steps. 

\medskip
\textbf{\emph{Step 1.}} \label{step1Sec4.1} First, let $\delta>0$ be small  in a way which will become clear later, and define the set
\begin{equation}\label{eq:Udelta}
	U_{\delta}=\big\{m \in [-1,1]: f_{\beta}(m)\leq f_{\beta}(m_-)+\delta  \big\}.
\end{equation}
We write $U_{\delta}^c= [-1,1] \setminus U_{\delta}$ and we denote by $U_{\delta}(m)$ the connected component of $U_{\delta}$ 
containing $m$. Note that $\{m_-,m_+\}\in U_{\delta}$.
In general, $U_{\delta}(m_-)$ and $U_{\delta}(m_+)$ may have non empty intersection, but we \emph{choose $\delta$} such that  
$m^* \notin U_{\delta}$, implying that $U_{\delta}$ is partitioned by the \emph{disjoint} sets $U_{\delta}(m_-)$ and $U_{\delta}(m_+)$. 
For this to hold, it suffices to take $\delta < f_{\beta}(m^*) -f_{\beta}(m_-)$.
Moreover, we \emph{choose} $\delta$ also such that $-1\notin U_{\delta}(m_-)$. For this to hold, it suffices to take $\delta < f_{\beta}(-1) -f_{\beta}(m_-)$. Thus, we \emph{choose}  $\delta < \min{\big(f_{\beta}(-1), f_{\beta}(m^*)\big)} -f_{\beta}(m_-)$. These conditions are needed to prove \eqref{eq:part1} below.

Let us denote by $m_\delta$ the unique point in $(m^*, m_+)$ such that 
\begin{equation}
	f_{\beta}(m_{\delta}) = f_{\beta}(m_-)+\delta.
\end{equation}

\medskip
\textbf{\emph{Step 2.}} With  $\delta$  chosen as above, we  define a sequence $(\delta_N)_{N\in \NN}$,
converging to $ \delta$ from below, 
such that the left extreme of $U_{\delta_N}(m_+)$ is in $\GN$. Specifically, 
we define $\delta_N$ as follows: 
\begin{equation}\label{eq:deltaN}
	\delta_N= \max \left\{  \bar{\delta}\in(0, \delta] \,: \,\exists \, m \in  U_{\delta}(m_+) \cap \GN \setminus [m_+,1] \text{ s.t. } f_\beta(m)= f_\beta(m_-)+\bar{\delta} \right\},
\end{equation}
for $N$ sufficiently large.
Moreover,  set
\begin{equation}
	U_{\delta,N}=U_{\delta_N} \cap \GN \text{, } \qquad 
	U_{\delta,N}^c= \GN \setminus  U_{\delta,N}\qquad  \text{ and } \qquad  
	U_{\delta,N}(m)=U_{\delta_N}(m) \cap \GN, 
\end{equation}
for all $m \in [-1,1]$. 
Thus, we have the partitions
\begin{equation}
	\GN=U_{\delta,N}(m_-) \cup U_{\delta,N}(m_+) \cup U_{\delta,N}^c
\end{equation}
and
\begin{equation}\label{eq:partit_SN}
	\Ss_N =\Ss_N\left[U_{\delta,N}(m_-)\right] \cup \Ss_N\left[m_+(N)\right] \cup  
	\Ss_N\big[U_{\delta,N}^c\big] \cup \Ss_N\left[U_{\delta,N}(m_+)\setminus\{m_+(N)\}\right].
\end{equation}	
\begin{remark}
	Notice that, for $N$ sufficiently large, ${U_{\delta,N}(m_-(N))=U_{\delta,N}(m_-)}$ and $U_{\delta,N}(m_+(N))=U_{\delta,N}(m_+)$.  Furthermore, with these definitions, $m_{\delta_N} \in U_{\delta,N}$ and it is the left extreme of $ U_{\delta,N}(m_+)$. 
\end{remark}

\medskip
\textbf{\emph{Step 3.} }
Let $\varepsilon>0$ be arbitrarily small (the choice of $\varepsilon$ will be relevant in Section~\ref{sec:part3}). 
We denote by  $m_{\varepsilon}$ the only point in a small left neighbourhood of $m_+$, 
more precisely in $U_{\delta}(m_+)\setminus [m_+,1]$, such that
\begin{equation}\label{eq:mEps}
	f_{\beta}(m_{\varepsilon})=f_{\beta}(m_+) + \varepsilon.
\end{equation}
Let us define an $\varepsilon$-dependent parameter $\theta>0$ by
\begin{equation}\label{eq:deftheta}
	\theta=m_+-m_{\varepsilon}.
\end{equation}

\medskip

\textbf{\emph{Step 4.}} Similarly to Step 2, fixed $\varepsilon>0$, we want to define a sequence $(\varepsilon_N)_{N\in \NN}$ 
converging to $\varepsilon$ from below such that $m_{\varepsilon_N}$ is in $\GN$. More precisely, we define $\varepsilon_N$ as follows
\begin{equation}\label{eq:epsN}
	\varepsilon_N= \max \big\{  \bar{\varepsilon} \in(0,\varepsilon] \,: \,\exists \, m \in  U_{\delta,N}(m_+)  \setminus [m_+,1] \text{ s.t. } f_\beta(m)= f_\beta(m_+)+\bar{\varepsilon} \,\big\}.
\end{equation}
We will  use later that $m_{\varepsilon_N} \in U_{\delta,N}(m_+)$ and it satisfies $f_\beta(m_{\varepsilon_N})= f_\beta(m_+)+\varepsilon_N$.

Moreover, given $\varepsilon>0$, we define the sequence $(\theta_N)_{N \in \NN}$, analogously to \eqref{eq:deftheta}, by setting $\theta_N= m_+(N)-m_{\varepsilon_N}$. 
$\theta_N$ plays an important role in Lemma~\ref{lem:hitting6.11} below. 
Notice that $\lim_{N \to \infty} \theta_N = \theta$ and, if $m_+ \neq m_+(N)$, then ${f(m_{\varepsilon_N})-f(m_+(N))\neq \varepsilon_N}$.

\medskip

\begin{figure}
	\centering
	\begin{tikzpicture}
		\begin{axis}[
			xtick=\empty,
			ytick=\empty,
			extra y ticks={2.1192, 6.0356,-0.2798,4},
			extra y tick labels={$f_{\beta}(m_-)$, $$,$$,$f_{\beta}(m_-)+\delta$ }, 
			xmin=-1.1, xmax=1.1, 
			ymin=-1, 
			extra x ticks={-1,-0.936,-0.728, -0.418,-0.043, 0.303,0.7,0.772,1},
			extra x tick labels={-1 ,,  , , ,,,,1}, 
			]
			\addplot[blue, name path=curve,samples=300, smooth, domain=-1:1]
			plot (\x, { (16*\x^4 -18*\x^2-1.6*\x+6) }) ;
			\addlegendentry{$f_{\beta}$};
			
			\coordinate  (m1) at (axis cs: -0.936,-1);  
			\coordinate  (fm1) at (axis cs: -0.936,4);   	 
			\coordinate  (m-) at (axis cs: -0.728,-1);
			\coordinate  (fm-) at (axis cs: -0.728,2.1192);
			\coordinate  (m2) at (axis cs: -0.418,-1);  
			\coordinate  (fm2) at (axis cs: -0.418,4);  
			\coordinate (m*) at (axis cs: -0.043,-1);
			\coordinate (fm*) at (axis cs: -0.043,6.0356);
			\coordinate (mDelta) at (axis cs: 0.303,-1);
			\coordinate (fmDelta) at (axis cs: 0.303,4);
			\coordinate (mTheta) at (axis cs: 0.7,-1);
			\coordinate (fmTheta) at (axis cs: 0.7,-0.105);
			\coordinate  (m+) at (axis cs: 0.772,-1);
			\coordinate  (fm+) at (axis cs: 0.772,-0.2798);
			\coordinate (plus1) at (axis cs: 1,-1);
			
			\draw [red, line width=8pt] (m1) -- (m2); 
			\draw [red, line width=8pt] (mDelta) -- (plus1) ; 
			
			\draw [dashed] (m-) --  (fm-); 
			\draw [dashed] (m*) --  (fm*); 
			\draw [dashed] (m+) --  (fm+); 
			\draw (m1) --  (fm1); 
			\draw (m2) --  (fm2); 
			\draw (mDelta) --  (fmDelta); 
			\draw (mTheta) --  (fmTheta); 
			
			\draw [dashed, red, name path=plus_delta] (axis cs: -1,4) -- (axis cs: 1,4); 
			
		\end{axis}
		\foreach \m/\name in {m-/$m_-$,mDelta/$m_{\delta}$, m+/$m_+$} 
		{\node [anchor=north] at ([yshift=-3pt]\m) {\name};}  
		\node [anchor=north east] (labelmE) at ([yshift=-3pt] mTheta) {$m_{\varepsilon}$};  
		\node [anchor=north] at (m*) {$m^*$};  
		\node [outer sep=3.2pt] (pointmE) at (mTheta) {};
		\draw [<-] (labelmE) -- (pointmE);
		\draw [decorate,decoration={brace,amplitude=8pt,mirror}]
		([yshift=-16pt] m1) -- ([yshift=-16pt] m2) node [black,midway, below=8pt] {\footnotesize	$U_{\delta}(m_-)$};
		\draw [decorate,decoration={brace,amplitude=8pt,mirror}]
		([yshift=-16pt] mDelta) -- ([yshift=-16pt] plus1) node [black,midway, below=8pt] {\footnotesize	$U_{\delta}(m_+)$};
	\end{tikzpicture}
	\caption{Graph of $f_{\beta}$ and decomposition  of the magnetisation space $[-1,1]$: 
		the two intervals $U_{\delta}(m_-)$
		and $U_{\delta}(m_+)$ around the two minima are drawn, together with the special points 
		$m_{\delta}$, $m_{\varepsilon}$. $U_{\delta}$ is painted in red. } \label{fig:wells}
\end{figure}
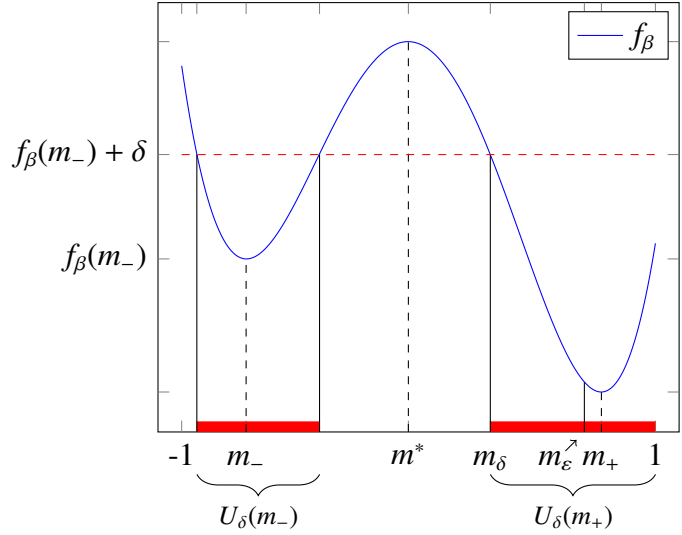

\subsection{Upper bound on the harmonic sum} \label{sec:upperHarm}
In this section we prove the first part of Theorem~\ref{thm:upperHarm} by giving an upper bound 
on the harmonic sum in \eqref{eq:harmonicsum}.

We will estimate the contribution of each set of the partition in \eqref{eq:partit_SN} to the sum in \eqref{eq:harmonicsum}. As one expects, the only relevant contribution will be given by the terms in $ \Ss_N[U_{\delta,N}(m_-)]$. Indeed,  $ \mu_{\beta,N}$  is very small in $\Ss_N[U_{\delta,N}^c]$ while $ h^N_{m_-,m_+}$   is very small in $\Ss_N[U_{\delta,N}(m_+)]$ and we will see the two contributions on these two sets turn out to be irrelevant. 

The main ingredients in the proof of the upper bound are Corollary~\ref{cor:sumQ} 
and  Lemma~\ref{lem:hittingP} below. The proof of the latter result is quite technical and it is postponed to Section~\ref{sec:lemmas_hitting}.

\begin{proof}[Proof of Theorem~\ref{thm:upperHarm}. Upper bound.]
	We are ready to start estimating the contributions of each disjoint set of the partition in \eqref{eq:partit_SN} to the sum in \eqref{eq:harmonicsum}.
	
	\medskip\noindent
	\textbf{Part 1.  Sum on $ \Ss_N[U_{\delta,N}(m_-)]$.}
	This will be the relevant part. 
	Using first that  $h^N_{m_-,m_+}(\si)\leq1$, \eqref{eq:defQ} and  \eqref{eq:sumQUpp2} of Corollary~\ref{cor:sumQ} with $\bar{g}(m)=\mathds{1}_{m\in U_{\delta,N}(m_-)}(m)$
	we obtain 
	\begin{equation}\label{eq:part1}
		\begin{split}
			&\sum_{\si \in \Ss_N[U_{\delta,N}(m_-)] } \mu_{\beta,N}(\si)  \,  h^N_{m_-,m_+}(\si) 
			\leq \sum_{m \in U_{\delta,N}(m_-)} \Qq_{\beta,N}(m) \\
			&\overset{P(s)}{\leq}\frac{ \eee^{s+\alpha}\left(1+o(1)\right)}{Z_{\beta,N}} 
			\sum_{m \in U_{\delta,N}(m_-)}  \exp{\left(-\beta N f_{\beta} (m)\right)} \sqrt{\frac{2}{\pi N (1-m^2)}}\\
			&=  \frac{\eee^{s+\alpha }\left(1+o(1)\right)\exp{\left(-\beta N f_{\beta} (m_-)\right)}}{Z_{\beta,N} \sqrt{(1-m_-^2) \, \beta  f''_{\beta} (m_-) }}.\\
		\end{split}
	\end{equation}
	In the second line we used our assumption of $\delta_N$ being small enough such that $-1\notin U_{\delta,N}(m_-)$ (see Section~\ref{step1Sec4.1}, Step 1).
	To obtain the last equality  we first approximated, for $N$ sufficiently large, the sum with an integral and then applied the saddle point method (see, for instance, de Bruijn~\cite[Chp 5.7]{Bru61}),  where  $m_-$ 
	is the maximum point of $-\beta f_{\beta}$ on the considered domain. Notice that here we use the fact that $m^* \notin U_{\delta,N}(m_-)$, which holds again for $\delta_N$ small enough (see Section~\ref{step1Sec4.1}, Step 1).
	More precisely, 
	\begin{equation}\label{eq:useLapl}
		\begin{split}
			\sum_{m \in U_{\delta,N}(m_-)}& \exp\Big(-\beta N f_{\beta} (m)\Big)\frac{1}{\sqrt{ (1-m^2)}} \\
			&\approx 
			\frac{N}{2} \int_{a}^{b}  \exp{\Big(-\beta N f_{\beta} (x)\Big)}\frac{1}{ \sqrt{(1-x^2)}} \dd x  \\
			&= \exp\Big(-\beta N f_{\beta} (m_-)\Big) \frac{1}{ \sqrt{(1-m_-^2) }} \sqrt{\frac{\pi N }{2\beta  f''_{\beta} (m_-) }}\left(1+o(1)\right),
		\end{split}
	\end{equation}
	where $-1<a,b \in \GN$ are the left and right extremes of $U_{\delta,N}(m_-)$, respectively.

	\smallskip
	\noindent
	\textbf{ Part 2.  Sum on $\Ss_N[m_+(N)]$.}
	Being by definition $ h^N_{m_-,m_+}(\si)=0$ 
	for all $\si \in \Ss_N[m_+(N)]$, we trivially have 
	\begin{equation}\label{eq:onm+}
		\sum_{\si \in \Ss_N[m_+(N)]} \mu_{\beta,N}(\si) \,  h^N_{m_-,m_+}(\si)= 0.
	\end{equation}
	
	\noindent
	\textbf{ Part 3.  Sum on $ \Ss_N[U_{\delta,N}^c]$.}
	
	Using $h^N_{m_-,m_+} \leq 1$ and  \eqref{eq:defQ}, we have
	\begin{equation}\label{eq:part2}
		\begin{split}
			&\sum_{\si \in \Ss_N[U_{\delta,N}^c]} \mu_{\beta,N} (\si) \,  h^N_{m_-,m_+}(\si) \leq \sum_{\si \in \Ss_N[U_{\delta,N}^c]} \mu_{\beta,N} (\si)  = \sum_{m \in U_{\delta,N}^c} \Qq_{\beta,N} (m)\\
			& = \sum_{m \in U_{\delta,N}^c\setminus\{1,-1 \}} \Qq_{\beta,N} (m) + \sum_{m \in U_{\delta,N}^c \cap \{1,-1 \}} \Qq_{\beta,N} (m)
			. \\
		\end{split}
	\end{equation} 
	
	We bound the right hand side using \eqref{eq:sumQUpp2} of Corollary~\ref{cor:sumQ} with $\bar{g}(m)=\mathds{1}_{m\in U_{\delta,N}^c}(m)$ obtaining
	\begin{equation}
		\begin{aligned}
			&\sum_{\si \in \Ss_N[U_{\delta,N}^c]} \mu_{\beta,N} (\si)\,  h^N_{m_-,m_+}(\si) \\
			&	\overset{P(s)}{\leq} \frac{ \eee^{s+\alpha}\left(1+o(1)\right)}{Z_{\beta,N}} \sum_{\substack{m \in U_{\delta,N}^c\setminus\{1,-1 \}}}\exp{\Big(-\beta N f_{\beta} (m)\Big) }\sqrt{\frac{2}{\pi N (1-m^2)}} \\
			&	\qquad
			+ \frac{ \eee^{s+\alpha}\left(1+o(1)\right)}{Z_{\beta,N}} \sum_{m \in U_{\delta,N}^c \cap \{1,-1 \}}\exp\Big(-\beta N f_{\beta} (m)\Big) \\
			& \leq \frac{\eee^{s+\alpha }\left(1+o(1)\right) }{Z_{\beta,N}}
			\exp{\Big(-\beta N  \left(f_{\beta} (m_-) + \delta_N \right)\Big)}
			\left(\sqrt{\frac{2}{\pi N}} \sum_{m \in U_{\delta,N}^c \setminus\{1,-1 \}} \frac{1}{\sqrt{(1-m^2)}} + 2 \right), 
		\end{aligned}
	\end{equation}
	where in the last inequality we used the bound $f_{\beta}(m)\geq f_{\beta} (m_-) + \delta_N$ given by the definition of $ U_{\delta,N}^c$ (see \eqref{eq:Udelta}).
	
	\smallskip\noindent
	\textbf{ Part 4.   Sum on $\Ss_N[U_{\delta,N}(m_+)\setminus\{m_+(N)\}]$.} \label{sec:part3}
	Using \eqref{eq:h_prop} and the fact that, for any $\si \in \Ss_N$ such that $m(\si)>m_+(N)$,
	$\PP_{\si}\left(\tau_{\Ss_N[m_-(N)]} < \tau_{\Ss_N[m_+(N)]}\right)= 0$, 
	we get 
	\begin{equation}\label{eq:part3}
		\begin{split}
			\sum_{{\si \in \Ss_N[U_{\delta,N}(m_+)\setminus\{m_+(N)\}]}} \hspace{-25pt}\mu_{\beta,N}(\si) \,  h^N_{m_-,m_+}(\si) 
			=	\sum_{\si \in \Ss_N\left[[m_{\delta_N},m_+(N))\right]} \hspace{-13pt}\mu_{\beta,N} (\si) \,\PP_{\si}\left(\tau_{\Ss_N[m_-(N)]} < \tau_{\Ss_N[m_+(N)]}\right).
		\end{split}
	\end{equation}
	Thus, applying Lemma~\ref{lem:hittingP} below, the following holds for any $\gamma \in (0,1)$
	\begin{equation} \label{eq:Um+}
		\begin{split}
			&\sum_{\si \in \Ss_N[U_{\delta,N}(m_+)\setminus\{m_+(N)\}]} \mu_{\beta,N}(\si)\, h^N_{m_-,m_+}(\si) \\
			&\leq   \exp{\Big(-\beta N (1-\gamma)f_{\beta}(m_-)\Big)}\sum_{m \in [m_{\delta_N},m_+(N))} \Qq_{\beta,N} (m)\,
			\Big[ \exp{\Big(\beta N (1-\gamma)f_{\beta}(m)\Big)}\Bigr.\\
			&\qquad \Bigl.+\exp{\Big(\beta N (1-\gamma)\left(f_{\beta}(m_+)+3\,\varepsilon_N\right)+N\ell_N(\theta_N)\Big)}\Big]\exp{\Big(-\beta N (1-\gamma)\, \delta_N \Big)} (1+o(1)).
		\end{split}
	\end{equation} 
	We use \eqref{eq:sumQUpp2} of Corollary~\ref{cor:sumQ} with $\bar{g}(m)$ defined by
	\begin{equation}
		\bar{g}(m)=\left[ \exp{\Big(\beta N (1-\gamma)f_{\beta}(m)\Big)} +\exp{\Big(\beta N (1-\gamma)\left(f_{\beta}(m_+)+3\,\varepsilon_N\right)+N\ell_N(\theta_N)\Big)}\right],
	\end{equation}
	for $m \in [m_{\delta_N},m_+(N))$ and $\bar{g}(m)=0$ for $m \in \GN\setminus[m_{\delta_N},m_+(N))$.
	Thus, we obtain 
	\begin{equation}\label{eq:better_b}
		\begin{split}
			&\sum_{\si \in \Ss_N[U_{\delta,N}(m_+)\setminus\{m_+(N)\}]} \mu_{\beta,N}(\si) \,  h^N_{m_-,m_+}(\si)  \\
			&\overset{P(s)}{\leq} \frac{\eee^{s+\alpha}\left(1+o(1)\right) }{Z_{\beta,N}}
			\exp{\Big(-\beta N (1-\gamma)\left(f_{\beta}(m_-)+\delta_N\right)\Big)}
			\sum_{m \in [m_{\delta_N},m_+(N))} \exp{\Big(-\beta N f_{\beta} (m)\Big)} \\
			& \, \, \,\times \sqrt{\frac{2}{\pi N (1-m^2)}}   \Big[\exp{\Big(\beta N (1-\gamma)f_{\beta}(m)\Big)} +
			\exp{\Big(\beta N (1-\gamma)\left(f_{\beta}(m_+)+3\varepsilon_N\right)+N\ell_N(\theta_N)\Big)}
			\Big]\\
			& \leq \frac{\eee^{s+\alpha} \left(1+o(1)\right)}{Z_{\beta,N}} \exp{\Big(-\beta N (1-\gamma)\left(f_{\beta}(m_-)+\delta_N\right)\Big)} \sqrt{\frac{2N}{\pi (1-m_+^2)}} \\
			&\, \, \, \,\times  
			\Big[\exp{\Big(-\gamma \,\beta N f_{\beta} (m_+)\Big)} + 
			\exp{\Big(\beta N (1-\gamma)\left(f_{\beta}(m_+)+3\varepsilon_N\right)+N\ell_N(\theta_N) -\beta N f_{\beta} (m_+)\Big)} \Big]\\
			&= \frac{\eee^{s+\alpha}\left(1+o(1)\right) }{Z_{\beta,N}}\exp{\Big(-\beta N f_{\beta} (m_-)\Big)} \sqrt{\frac{2N}{\pi  (1-m_+^2)}} \exp{\Big(-\gamma \,\beta N[ f_{\beta} (m_+)-f_{\beta} (m_-)] \Big)} \\
			& \, \, \, \, \times   \exp{\Big(-\beta N (1-\gamma)(\delta_N -3\,\varepsilon_N) +N\ell_N(\theta_N)\Big)}
			\Big[ \exp{\Big(-\beta N (1-\gamma)3\,\varepsilon_N -N\ell_N(\theta_N)\Big)}+1\Big] 
			\\
			&\leq \frac{\eee^{s+\alpha}\left(1+o(1)\right)}{Z_{\beta,N}} \exp{\Big(-\beta N f_{\beta} (m_-)\Big)} \sqrt{\frac{2}{\pi  (1-m_+^2)}} \\
			&\, \, \, \,\times  
			\exp\left[-\beta N\left(\gamma\,[ f_{\beta} (m_+)-f_{\beta} (m_-)] +(1-\gamma)(\delta_N -3\,\varepsilon_N) -\tfrac{1}{\beta}\ell_N(\theta_N) -\varepsilon_N \right)\right].
		\end{split}
	\end{equation}
	In the last step we embedded $[\,\exp{\big(-\beta N (1-\gamma)3\,\varepsilon_N -N\ell_N(\theta_N)\big)}+1]$   in the already present $(1+o(1))$ 
	and bounded $\sqrt{N}$ by $\exp{\big(-\beta N (-\varepsilon_N)\big)}$, because for $N$ large enough $\tfrac{\log(N)}{2\beta N} \leq \varepsilon_N$ (which converges to $\varepsilon>0$, see Step 4 in Section \ref{sec:notat_delta_theta}). 
	
	Now we prove that this part is not relevant compared to the right hand side of \eqref{eq:part1}. In particular, we  show that,
	for a certain choice of $\gamma$,
	\begin{equation}
		c_N=\gamma [f_{\beta}(m_+) -f_{\beta}(m_-)] +(1-\gamma)( \delta_N -3\,\varepsilon_N)-\tfrac{1}{\beta}\ell_N(\theta_N) -\varepsilon_N
	\end{equation}
	is positive and its limit,
	\begin{equation} 
		\lim_{N\to \infty} c_N=\gamma \,[f_{\beta}(m_+) -f_{\beta}(m_-)] +(1-\gamma)( \delta -3\,\varepsilon)-\tfrac{\theta}{2\beta}\big(  \log (2) +3 - \log(1-m_+)  \big) -\varepsilon,
	\end{equation}
	is positive and finite.	In order to achieve this, we  choose $\gamma \in (0,1)$ small enough, such that $c_N$ and its limit are positive, definitely in $N$. 
	In particular, we want to impose
	\begin{equation}\label{eq:gammaConstr}
		0<\gamma < \frac{\delta_N -4\,\varepsilon_N -\tfrac{1}{\beta}\ell_N(\theta_N)}{f_{\beta}(m_-) -f_{\beta}(m_+) +\delta_N -3\,\varepsilon_N}<1,
	\end{equation}
	definitely in $N$, and
	\begin{equation}\label{eq:gammaConstr_1}
		0<\gamma < \frac{\delta -4\,\varepsilon -\tfrac{1}{\beta}\lim_{N \to \infty} \ell_N(\theta)}{f_{\beta}(m_-) -f_{\beta}(m_+) +\delta -3\,\varepsilon}<1.
	\end{equation}
	
	First, we notice that it is easy to check that the previous quantities are strictly smaller than $1$. Second, we want to show that 
	a strictly positive $\gamma$ satisfying \eqref{eq:gammaConstr}-\eqref{eq:gammaConstr_1}  exists.
	Note that $\ell_N(\theta_N)$, defined in \eqref{eq:elltheta}, has the following trivial upper bound for every $N$,
	\begin{equation}
		\ell_N(\theta_N)
		\leq \theta_N\left(\beta + \log 2 + O(\theta_N)\right).
	\end{equation}
	Thus, a sufficient condition is to choose, for $N$ large enough, $\gamma\geq \gamma_0$,
	where
	\begin{equation}\label{eq:gammaConstr_2}
		\gamma_0 =	
		\frac{\delta -4\,\varepsilon - \theta\left(1 +\tfrac{\log 2}{\beta} +O(\theta)\right)}{f_{\beta}(m_-) -f_{\beta}(m_+) +\delta}
	\end{equation}
	is clearly strictly positive. Indeed, we can choose $\varepsilon>0$ sufficiently small for the numerator on the left hand side of \eqref{eq:gammaConstr_2} to be positive, while $\theta$ is  small accordingly to $\varepsilon$ (see Section \ref{sec:notat_delta_theta}).
	We conclude by obtaining,
	for $N$ sufficiently large,
	\begin{multline}\label{eq:upper_part4}
		\sum_{\si \in \Ss_N[U_{\delta,N}(m_+)\setminus\{m_+(N)\}]}  \mu_{\beta,N}(\si) \,  h^N_{m_-,m_+}(\si)
		\overset{P(s)}{\leq} \frac{ \eee^{s+\alpha + o(1)}}{Z_{\beta,N}} \exp{\left(-\beta N (f_{\beta} (m_-)+c_N)\right)}  \sqrt{\tfrac{2}{\pi (1-m_+^2)}},
	\end{multline}
	where $0<c_N=O(1)$.
	
	\medskip\noindent
	\textbf{ Conclusion.}
	
	With the previous bounds at hand, we are now ready to conclude the proof of the upper bound. 
	Decomposing the sum over $\Ss_N$ using \eqref{eq:partit_SN}, and inserting the estimates we computed above
	into \eqref{eq:harmonicsum}, we obtain
	
	\begin{equation}\label{eq:sumUp}
		\begin{split}
			&\sum_{\si \in \Ss_N} \mu_{\beta,N}(\si) \,  h^N_{m_-,m_+}(\si) \\
			&\overset{P(s)}{\leq} \frac{\eee^{s+\alpha }\left(1+o(1)\right)}{Z_{\beta,N}}  \exp{\Big(-\beta N f_{\beta} (m_-)\Big)} \left[\eee^{-\beta N  \delta_N} 
			\left(\sqrt{\frac{2}{\pi N}} \sum_{m \in U_{\delta,N}^c \setminus\{ 1,-1 \}} \frac{1}{\sqrt{(1-m^2)}} + 2 \right)   \right.\\
			&\qquad \left. + \sqrt{\frac{2}{\pi \,(1-m_+^2)}}   \eee^{-\beta N c_N}  
			+ \frac{1}{ \sqrt{(1-m_-^2) \,\beta  f''_{\beta} (m_-) }}\right]\\
			& \leq \frac{\eee^{s+\alpha}}{Z_{\beta,N}}  \exp{\Big(-\beta N f_{\beta} (m_-)\Big)} \frac{1}{ \sqrt{(1-m_-^2) \, \beta  f''_{\beta} (m_-) }}\left(1+o(1)\right),
		\end{split}
	\end{equation}  
	concluding the proof.	
\end{proof}

\subsection{Some technical results} \label{sec:lemmas_hitting}
In this section we prove  Lemma~\ref{lem:hittingP}, which is pivotal in obtaining
the upper bound in Theorem~\ref{thm:upperHarm} (see \eqref{eq:Um+}).  The proof is quite involved, therefore we split it 
into subsequent technical results. 
Before starting the proof, we give a brief outline of this section. First, we state Lemma~\ref{lem:hittingP} and 
prove it via Lemmas~\ref{lem:hitting6.11}, \ref{lem:for_mTheta} and \ref{lem:hitting6.3}, which follow later on. 
Second, we give the proof of Lemmas~\ref{lem:hitting6.11} and \ref{lem:for_mTheta}.
The latter relies on Lemma~\ref{lem:hitting6.3}, which we subsequently prove using Lemma~\ref{lem:superHarm}.
We conclude the section proving Lemma~\ref{lem:superHarm}. 
Throughout this section we will use the notation introduced in Section~\ref{sec:notat_delta_theta}.

\begin{lemma}\label{lem:hittingP}
	For all $\sigma \in \Ss_N\left[[m_{\delta_N},m_+(N))\right]$, for all $\gamma \in (0,1)$ and $\varepsilon>0$, 
	\begin{equation}\label{eq:hittingP}
		\begin{split}
			\mathbb{P}_{\sigma}& \Big(\tau_{\Ss_N[m_-(N)]}<\tau_{\Ss_N[m_+(N)]}\Big)  \leq  \exp{\Big(-\beta N (1-\gamma)[f_{\beta}(m_-)+\delta_N]\Big)} (1+o(1)) \\
			&\qquad \times \Bigg[ \exp{\Big(\beta N (1-\gamma)f_{\beta}(m(\si))\Big)} +  
			\exp{\Big(\beta N (1-\gamma)[f_{\beta}(m_+)+3\varepsilon_N] +N\ell_N(\theta_N) \Big)}\Bigg],
		\end{split}
	\end{equation}
	where $\ell_N(\cdot)$ is defined in \eqref{eq:elltheta}.
\end{lemma}

\begin{proof}
	For all $\sigma \in \Ss_N\left[[m_{\delta_N},m_+(N))\right]$, we have
	\begin{equation}
		\begin{split}
			&\mathbb{P}_{\sigma} \Big(\tau_{\Ss_N[m_-(N)]}<\tau_{\Ss_N[m_+(N)]}\Big)   \\
			&= \mathbb{P}_{\sigma} \left(\tau_{\Ss_N[m_-(N)]}<\tau_{ \Ss_N[m_+(N)]}, \tau_{\Ss_N[m_-(N)]}<\tau_{\Ss_N[m_{\varepsilon_N}]}\right) \\
			& \qquad + \sum_{\eta \in \Ss_N[m_{\varepsilon_N}] } \PP_{\si}\left(\tau_{\Ss_N[m_-(N)]}<\tau_{\Ss_N[m_+(N)]}, \tau_{\eta} < \tau_{ \Ss_N[\{m_{\varepsilon_N},m_-(N),m_+(N)\}]  }\right) \\
			& = \mathbb{P}_{\sigma} \left(\tau_{\Ss_N[m_-(N)]}<\tau_{ \Ss_N[\{m_+(N), m_{\varepsilon_N}\}]}\right)    \\
			&\qquad + \sum_{\eta \in \Ss_N[m_{\varepsilon_N}] } \PP_{\si}\left(\tau_{\Ss_N[m_-(N)]}<\tau_{\Ss_N[m_+(N)]}\,
			\Big| \,\tau_{\eta} < \tau_{ \Ss_N[\{m_{\varepsilon_N},m_-(N),m_+(N)\}]}\right) \\
			&\qquad \qquad \times   \PP_{\si}\left(\tau_{\eta} < \tau_{ \Ss_N[\{m_{\varepsilon_N},m_-(N),m_+(N)\}]  }\right), 
		\end{split}
	\end{equation} 
	where we notice that,
	\begin{equation}
		\begin{split}
			\PP_{\si}&\left(\tau_{\Ss_N[m_-(N)]}<\tau_{\Ss_N[m_+(N)]}\,\Big| \, \tau_{\eta} < \tau_{ \Ss_N[\{m_{\varepsilon_N},m_-(N),m_+(N)\}]}\right) \\
			&	\quad= \PP_{\eta}\Big(\tau_{\Ss_N[m_-(N)]}<\tau_{\Ss_N[m_+(N)]}\Big).
		\end{split}
	\end{equation}
	Using the Markov property and taking the maximum of the first factor out of the sum, we have that, 
	for all $\sigma \in \Ss_N\big[[m_{\delta_N},m_+(N))\big]$,
	\begin{equation}\label{eq:allSigma}
		\begin{split}
			&\mathbb{P}_{\sigma} \Big(\tau_{\Ss_N[m_-(N)]}<\tau_{\Ss_N[m_+(N)]}\Big)   \\
			& \leq \mathbb{P}_{\sigma} \left(\tau_{\Ss_N[m_-(N)]}<\tau_{ \Ss_N[\{m_+(N), m_{\varepsilon_N}\}]}\right) \\
			&\quad+ \left(\max_{\eta \in \Ss_N[m_{\varepsilon_N}] } \PP_{\eta} \left(\tau_{\Ss_N[m_-(N)]}<\tau_{\Ss_N[m_+(N)]}\right)  \right)
			\sum_{\eta \in \Ss_N[m_{\varepsilon_N}] } \PP_{\si}\left(\tau_{\eta} < \tau_{ \Ss_N[\{m_{\varepsilon_N},m_-(N),m_+(N)\}]  }\right) \\
			& = \mathbb{P}_{\sigma} \left(\tau_{\Ss_N[m_-(N)]}<\tau_{ \Ss_N[\{m_+(N), m_{\varepsilon_N}\}]}\right) \\
			& \quad + \left(\max_{\eta \in \Ss_N[m_{\varepsilon_N}] } \PP_{\eta} \left(\tau_{\Ss_N[m_-(N)]}<\tau_{\Ss_N[m_+(N)]}\right)  \right)
			\PP_{\si}\left(\tau_{\Ss_N[m_{\varepsilon_N}]} < \tau_{ \Ss_N[\{m_-(N),m_+(N)\}]  }\right) \\
			& \leq \mathbb{P}_{\sigma} \left(\tau_{\Ss_N[m_-(N)]}<\tau_{ \Ss_N[\{m_+(N), m_{\varepsilon_N}\}]}\right) \\
			&\quad+ \left(\max_{\eta \in \Ss_N[m_{\varepsilon_N}] } \PP_{\eta} \left(\tau_{\Ss_N[m_-(N)]}<\tau_{\Ss_N[m_+(N)]}\right)  \right)
			\PP_{\si}\left(\tau_{\Ss_N[m_{\varepsilon_N}]} < \tau_{ \Ss_N[m_+(N)]  }\right). \\
		\end{split}
	\end{equation}

	We first consider the case  $\sigma \in \Ss_N[m_{\varepsilon_N}]$. By Lemma~\ref{lem:hitting6.11}, we get
	\begin{equation}\label{eq:use611}
		\begin{split}
			\mathbb{P}_{\sigma}& \Big(\tau_{\Ss_N[m_-(N)]}<\tau_{\Ss_N[m_+(N)]}\Big)   
			\leq  \mathbb{P}_{\sigma} \left(\tau_{\Ss_N[m_-(N)]}<\tau_{ \Ss_N[\{m_+(N), m_{\varepsilon_N}\}]}\right) \\
			&+ \left(\max_{\eta \in \Ss_N[m_{\varepsilon_N}] } \PP_{\eta} \left(\tau_{\Ss_N[m_-(N)]}<\tau_{\Ss_N[m_+(N)]}\right)  \right) 	\left(1- \eee^{- N \ell_N(\theta_N)}(1+o(1))\right). \\
		\end{split}
	\end{equation} 
	Taking the maximum over $\si$ and noticing that the same term appears in both right and left hand side of the inequality, we obtain
	\begin{equation} 
		\begin{split}
			&\max_{\si \in \Ss_N[m_{\varepsilon_N}]} \mathbb{P}_{\sigma} \Big(\tau_{\Ss_N[m_-(N)]}<\tau_{\Ss_N[m_+(N)]}\Big)   \\ 
			&	\leq \max_{\si \in \Ss_N[m_{\varepsilon_N}]} \mathbb{P}_{\sigma} \left(\tau_{\Ss_N[m_-(N)]}<\tau_{ \Ss_N[\{m_+(N), m_{\varepsilon_N}\}]}\right)\,
			\eee^{N \ell_N(\theta_N)}\,(1+o(1)) \\
			& \leq \exp\Big(-\beta N \, (1-\gamma)\left[f_{\beta}(m_-) + \delta_N-f_{\beta}\left(m_{\varepsilon_N}-\tfrac{2}{N}\right)  -\varepsilon_N\right]-N\ell_N(\theta_N) \Big)(1+o(1)), \\
		\end{split}
	\end{equation}	
	where we used Lemma~\ref{lem:for_mTheta}.
	
	By Taylor expansion of $f_{\beta}\left(m_{\varepsilon_N}-\tfrac{2}{N}\right)$ and definition of $m_{\varepsilon_N}$,  we get
	\begin{equation} \label{eq:maxTheta}
		\begin{split}
			&\max_{\si \in \Ss_N[m_{\varepsilon_N}]} \mathbb{P}_{\sigma} \Big(\tau_{\Ss_N[m_-(N)]}<\tau_{\Ss_N[m_+(N)]}\Big)   \\
			& \leq   \exp\Big(-\beta N \, (1-\gamma)\left[f_{\beta}(m_-) + \delta_N- 3\varepsilon_N -f_{\beta}(m_+) \right]-N\ell_N(\theta_N)\Big)(1+o(1)),
		\end{split}
	\end{equation}	
	where the last inequality holds for $N$ sufficiently large. Here we bounded the Taylor expansion error $O\left(\tfrac{1}{N}\right)$ with $\varepsilon_N$, which converges to $\varepsilon>0$ (see Step 4 in Section \ref{sec:notat_delta_theta}).

	Now we  consider the case where $\sigma \in \Ss_N\big[[m_{\delta_N},m_+(N)) \setminus \{m_{\varepsilon_N}\}\big]$. Going back to \eqref{eq:allSigma} and using again \eqref{eq:maxTheta}, we obtain
	\begin{equation}\label{eq:last_step}
		\begin{split}
			\mathbb{P}_{\sigma} &\Big(\tau_{\Ss_N[m_-(N)]}<\tau_{\Ss_N[m_+(N)]}\Big)
			\leq \mathbb{P}_{\sigma} \left(\tau_{\Ss_N[m_-(N)]}<\tau_{ \Ss_N[\{m_+(N), m_{\varepsilon_N}\}]}\right)   \\
			& 	\qquad+ \exp\Big(-\beta N (1-\gamma)\left[f_{\beta}(m_-) + \delta_N - 3\varepsilon_N-f_{\beta}(m_+) \right]-N\ell_N(\theta_N)\Big)(1+o(1))	\\
			& \leq \exp{\Big(-\beta N (1-\gamma)[f_{\beta}(m_-)+\delta_N]\Big)}(1+o(1)) \\
			&\qquad \times \Bigg[ \exp{\Big(\beta N (1-\gamma)f_{\beta}(m(\si))\Big)} +  
			\exp{\Big(\beta N (1-\gamma)[f_{\beta}(m_+)+3\varepsilon_N] +N\ell_N(\theta_N) \Big)}\Bigg].
		\end{split}
	\end{equation}  
	In the last inequality we used Lemma~\ref{lem:hitting6.3}, which holds for $\si \in \Ss_N[ [m_{\delta_N},m_{\varepsilon_N})]$, 
	and that $\mathbb{P}_{\sigma}\Big(\tau_{\Ss_N[m_-(N)]}<\tau_{ \Ss_N[\{m_+(N), m_{\varepsilon_N}\}]}\Big) =0$
	for all $\si \in \Ss_N[ (m_{\varepsilon_N},m_+(N))]$. 
	
\end{proof}

\begin{remark}
	In Lemma~\ref{lem:hittingP} one might try to further bound the right hand side of \eqref{eq:hittingP} using that ${f_{\beta}(m(\si))}$ is bounded by $f_{\beta}(m_{\delta_N})=f_{\beta}(m_-)+\delta_N$. 
	This would yield to the trivial upper bound $1$ on $\mathbb{P}_{\sigma} \big(\tau_{\Ss_N[m_-(N)]}<\tau_{\Ss_N[m_+(N)]}\big)  $, 
	which is not sufficient for our purpose of proving that the second term in \eqref{eq:sumUp} is negligible with respect to the last one. 
	The way to go is, therefore, to keep the dependence on $m(\si)$ in order to obtain later a more suitable bound, uniform in $m$,
	by exploiting the smallness of $\Qq_{\beta,N} (m(\si))$ in  \eqref{eq:Um+} and \eqref{eq:better_b}.
\end{remark}

In order for \eqref{eq:use611} to be true, we have to prove the following result. 
\begin{lemma}\label{lem:hitting6.11}
	For all $\si \in \Ss_N[m_{\varepsilon_N}]$, for $\varepsilon$ sufficiently small and for $N$ sufficiently large,
	\begin{equation}
		\PP_{\si}\Big(\tau_{\Ss_N[m_+(N)]}<\tau_{ \Ss_N[m_{\varepsilon_N}]} \Big) \geq \eee^{- N \ell_N(\theta_N)}(1+o(1)),
	\end{equation}
	where $\ell_N:\RR \to \RR$ is defined by
	\begin{multline}\label{eq:elltheta}
		\ell_N(x)= \tfrac{1}{2} \Big[ x  \,\big(\log 2 +\beta\,\abs{2-2h} +1 \big) - (1-m_+(N)+x) \log(1-m_+(N)+ x) \\ +(1-m_+(N)) \log(1-m_+(N)) \Big].
	\end{multline}
	
\end{lemma}

\begin{proof}	
	
	Recall that $\{\si(t)\}_{t\geq0}$ is the Markov chain 
	with transition probabilities \eqref{eq:rates} and, for $\si \in \Ss_N$ with $m(\si)<m_+(N)$, let
	\begin{multline}
		A_N(\si) = \Big\{ (\si(0), \si(1), \si(2),\dots) : \,  \si(0)=\si,  \forall i \in \NN,  \si(i) \in \Ss_N, \si(i)\sim \si(i+1), \\ 
		\exists \, k \in \NN \text{ s.t. } \si(k) \in \Ss_N[m_+(N)], 
		\text{ and } \forall i \leq k-1, \, \, m(\si(i+1))=m(\si(i))+\tfrac{2}{N}  	\Big\} 
	\end{multline}
	be the set of infinite paths  starting in  $\si$ and having increasing magnetisation until the set  $\Ss_N[m_+(N)]$ is reached. 
	
	Notice that, for fixed $\si$ and $N$, the number $k$ of steps of increasing magnetisation to reach  $\Ss_N[m_+(N)]$ is fixed, 
	namely $k=\frac{N}{2} (m_+(N)-m(\si))$.
	
	We want to partition $A_N(\si)$ according to the values of the first $k+1$ elements of its paths.
	Given a sequence $\pi \in \Ss_N^{k+1}$, let us denote by $\{\pi \}$ the set of all paths in $A_N(\si)$ in which the first $k+1$ elements are exactly given by $\pi$, namely
	\begin{equation}
		\{\pi \}= \big\{ \left(\si(0),\si(1), \dots,  \si(k), \si(k+1), \dots \right) \in A_N(\si): 
		(\si(0), \dots,  \si(k)) =\pi\big\}.
	\end{equation}
	Notice that, by definition of $A_N(\si)$, $\{\pi \}$ is empty for many $\pi \in \Ss_N^{k+1}$ . We denote by $B_N(\si)$ the set of all the sequences  $\pi \in \Ss_N^{k+1}$ such that $\{\pi\}$ is not empty. Thus, we obtain the following partition of $A_N(\si)$
	\begin{equation}
		A_N(\si) = \bigsqcup_{\pi \in B_N(\si)} \{\pi \}.
	\end{equation}
	
	Fix $\si \in \Ss_N[m_{\varepsilon_N}]$, then one simply notices that   
	
	\begin{equation}\label{eq:targetP}
		\begin{split}
			\PP_{\si}&\left(\tau_{\Ss_N[m_+(N)]}<\tau_{ \Ss_N[m_{\varepsilon_N}]} \right) 
			\geq \PP_{\si}(A_N(\si)) =
			\sum_{\pi \in B_N(\si)}\PP_{\si} \left( \{\pi \} \right).\\
	\end{split}\end{equation}
	Thus, we first find a lower bound on $\PP_{\si} \left( \{\pi \} \right)$ independent of $\pi$ in $B_N(\si)$ and later we compute the cardinality of $B_N(\si)$.
	Fix $\pi=(\si(0), \si(1), \si(2),\dots, \si(k)) \in B_N(\si)$, then we have
	\begin{equation}
		\begin{split}
			&\PP_{\si} \left( \{\pi\} \right)  = \prod_{i=1}^{k} p_N(\si(i-1), \si(i)) = 
			\frac{1}{N^k} \prod_{i=1}^{k} \exp\Big(-\beta \left[H(\si(i))-H(\si(i-1))\right]_+\Big)\\
			&\geq \frac{C^k}{N^k} \prod_{i=1}^{k}\exp\Big(-\beta N [E(m_i)-E(m_{i-1})]_+\Big)
			=\frac{C^k}{N^k} \prod_{i=1}^{k} \exp\left(-\beta \left[-2m_{i-1} -\frac{2}{N}-2h\right]_+\right),
		\end{split}
	\end{equation}
	where $m_i=m(\si(i))$, $C=\exp\left(-\beta \abs{2-2h}\right)$ and we used the following fact
	\begin{equation}
		\begin{split}
			&\frac{\exp\Big(-\beta[H(\si(i))-H(\si(i-1))]_+ \Big)}{\exp \Big(-\beta N[E(m_i)-E(m_{i-1})]_+ \Big)} 
			=\frac{\exp\Big(-\beta[H(\si(i))-H(\si(i-1))]_+ \Big)}{\exp \Big(-\beta \big[-2m_{i-1} -\frac{2}{N}-2h\big]_+ \Big)} \\
			&\geq \exp\Big(-\beta[H(\si(i))-H(\si(i-1))]_+ \Big) = \exp\left(-\beta \left[- \frac{2}{N} \sum_{j: j \neq r} J_{j r} \sigma(i-1)_j -2h \right]_+\right) \\
			& \geq \exp\left(-\beta \left[2-2h-\tfrac{2}{N}\right]_+\right)\geq \exp\Big(-\beta \abs{2-2h}\Big),
		\end{split}
	\end{equation}
	where $r$ is the index of the spin to be flipped to go from $\si(i-1)$ to $\si(i)$.
	Therefore, recalling that $m_i \in [m_{\varepsilon_N},m_+(N)]$, we obtain the following lower bound independent of $\pi$
	\begin{equation} \label{eq:P_pi}
		\begin{split}
			\PP_{\si} \left( \{\pi\} \right) \geq \frac{C^k}{N^k} \prod_{i=1}^{k} \exp\left(-\beta \left[-2m_{\varepsilon_N} -\frac{2}{N}-2h\right]_+\right) 
			=\frac{C^k}{N^k}.
		\end{split}
	\end{equation}
	Indeed, for $\varepsilon_N$ sufficiently small, $m_{\varepsilon_N}$ is close to $m_+(N)>0$, allowing us to assume ${m_{\varepsilon_N}>0}$. Therefore, $-2m_{\varepsilon_N} -\frac{2}{N}-2h<0$, which implies the last equality in \eqref{eq:P_pi}.
	
	We are left to compute the cardinality of $B_N(\si)$, with $\si \in \Ss_N[m_{\varepsilon_N}]$,   namely we have to count all  paths from $\si$ to $\Ss_N[m_+(N)]$ with increasing magnetisation and length $k+1$. Any of these paths is characterised by a final spin $\bar{\si} \in \Ss_N[m_+(N)]$ and a sequence of negative spins which are flipped. 
	Notice that $\bar{\si}$ is reachable by $\si$ through a path with increasing magnetisation if and only if the two following properties are satisfied: 
	$\bar{\si}$ has $k$ positive spins more than $\si$ and,
	for all $i \in [N]$,  $\si_i=+1$ implies $\bar{\si}_i=+1$.
	Thus, a configuration $\bar{\si} \in \Ss_N[m_+(N)]$ reachable by $\si$ through a path with increasing magnetisation is characterised by the $k$ spins which are negative in $\si$ and positive in $\bar{\si}$.
	Therefore, the number of reachable configurations $\bar{\si}$ is
	\begin{equation}
		{\binom{\frac{1}{2}N \left(1-m_{\varepsilon_N}\right)}{k}}
		={\binom{\frac{1}{2}N \left[1-m_+(N) +{\theta_N}\right]\,}{\frac{1}{2}N \theta_N}},
	\end{equation}
	being $\frac{1}{2}N (1-m_{\varepsilon_N})$ the number of negative spins of $\si \in \Ss_N[m_{\varepsilon_N}]$ and ${k=\frac{1}{2}N \theta_N}$, where $\theta_N$ has been defined in Section~\ref{sec:notat_delta_theta}.
	
	The number of paths with increasing magnetisation from $\si \in \Ss_N[m_{\varepsilon_N}]$ to a reachable ${\bar{\si} \in \Ss_N[m_+(N)]}$, both fixed, is $k!$, namely the number of permutations of the $k$ negative spins which are flipped along a path. 
	Thus, being $k=\frac{1}{2}N \theta_N$, the cardinality of $B_N(\si)$ is
	\begin{equation}
		\left(\tfrac{1}{2}N \theta_N\right)! \,
		{\binom{\frac{1}{2}N \big[1-m_+(N) +{\theta_N}\big]}{\frac{1}{2}N \theta_N}}.
	\end{equation}

	Going back to \eqref{eq:targetP}, we obtain
	\begin{equation}
		\begin{split}
			&\PP_{\si}\Big(\tau_{\Ss_N[m_+(N)]}<\tau_{ \Ss_N[m_{\varepsilon_N}]} \Big) 
			\geq  \sum_{\pi \in B_N(\si)} \PP_{\si} \big( \{\pi\} \big)  \\
			& \geq \left(\frac{C}{N}\right)^{\frac{1}{2}N \theta_N} \left(\tfrac{1}{2}N \theta_N \right)! \,
			{\binom{\frac{1}{2}N \big[1-m_+(N) +{\theta_N}\big]}{\frac{1}{2}N \theta_N}} \\
			& =  \eee^{- \frac{1}{2}N \theta_N \log \frac{N}{C}}
			\frac{N (1-m_+(N) +{\theta_N})}{2}!
			\left[\frac{N (1-m_+(N))}{2}!\right]^{-1}.
		\end{split}
	\end{equation}
	Using Stirling's approximation ${n!= \sqrt{2\pi n}\, \eee^{n(\log n -1)} (1+o(1))}$ and the notation  \begin{equation}
		{k_{\theta_N}=\frac{1-m_+(N) +\theta_N}{1-m_+(N)}},
	\end{equation}
	we obtain 
	\begin{multline}
		\frac{N (1-m_+(N) +\theta_N)}{2}!\left[\frac{N (1-m_+(N))}{2}!\right]^{-1}\\
		= \sqrt{k_{\theta_N}}\exp\left[ \tfrac{N (1-m_+(N))}{2}  \log (k_{\theta_N})  +  \tfrac{1}{2}N \theta_N
		\log \left( \tfrac{N (1-m_+(N) + \theta_N)}{2}\right)-\tfrac{1}{2}N \theta_N\right] (1+o(1)).
	\end{multline}
	Thus, since $k_{\theta_N}\geq 1$ and $C=\exp(-\beta \abs{2-2h})$, we conclude by
	\begin{eqnarray}
		\nonumber
		&&\PP_{\si}\left(\tau_{\Ss_N[m_+(N)]}<\tau_{ \Ss_N[m_{\varepsilon_N}]} \right) \\\nonumber
		&&\geq \sqrt{k_{\theta_N}}
		\mathrm{e}^{-\frac{N}{2} \left(   \theta_N \log \left(\frac{N}{C}\right) +\theta_N- (1-m_+(N)) \log (k_{\theta_N}) - \theta_N\log \left( \frac{N}{2}(1-m_+(N) + \theta_N)\right) \right)} (1+o(1))\\\nonumber
		&& \geq  
		\mathrm{e}^{-\frac{N}{2} \left(   \theta_N \log\left(\frac{N}{C}\right)  +\theta_N - (1-m_+(N)) \log (k_{\theta_N}) - \theta_N\log \left( \frac{N}{2}\right)  -\theta_N \log (1-m_+(N) + \theta_N) \right)} (1+o(1))\\
		\nonumber
		&& =  
		\mathrm{e}^{-\frac{N}{2} \left(   \theta_N \log (2)+  \theta_N \,\beta \abs{2-2h} +\theta_N 
			- (1-m_+(N) +\theta_N) \log (1-m_+(N) +\theta_N)\right)} \\\nonumber
		&&\qquad\times \,\mathrm{e}^{-\frac{N}{2} \left( 
			(1-m_+(N))\log(1-m_+(N))  \right)} (1+o(1)) \\
		&&= \eee^{-N\ell_N(\theta_N)}(1+o(1)).
	\end{eqnarray}
\end{proof}

To prove Lemma~\ref{lem:hittingP} we used the following fact. 
\begin{lemma}\label{lem:for_mTheta}
	For $\si \in \Ss_N[m_{\varepsilon_N}]$, for $N$ sufficiently large and any $\gamma \in (0,1)$, 
	\begin{multline}
		\PP_{\si}\left(\tau_{\Ss_N[m_-(N)]}<\tau_{ \Ss_N[\{m_+(N), m_{\varepsilon_N}\}]} \right) \\ 
		\leq \exp\Big(-\beta N (1-\gamma)
		\left[f_{\beta}(m_-) +  \delta_N-f_{\beta}\left(m_{\varepsilon_N}-\tfrac{2}{N}\right) -\varepsilon_N \right]\Big).
	\end{multline}
\end{lemma}

\begin{proof}
	Let us denote by $W_N(m)$ the event of making the first flip in $\Ss_N[m]$.
	
	For $\si \in \Ss_N[m_{\varepsilon_N}]$, conditioning on the first step, we obtain
	\begin{equation}
		\begin{split}
			\PP_{\si}&\left(\tau_{\Ss_N[m_-(N)]}<\tau_{ \Ss_N[\{m_+(N), m_{\varepsilon_N}\}]} \right) \\
			& =\PP_{\si}\left(W_N\left(m_{\varepsilon_N}+\tfrac{2}{N}\right)\right) \PP_{\si}\left(\tau_{\Ss_N[m_-(N)]}<\tau_{ \Ss_N[\{m_+(N), m_{\varepsilon_N}\}]} 
			\,\Big|\, W_N\left(m_{\varepsilon_N}+\tfrac{2}{N}\right) \right)  \\
			& \qquad  +  \PP_{\si}\left(W_N\left(m_{\varepsilon_N}-\tfrac{2}{N}\right)\right)
			\PP_{\si}\left(\tau_{\Ss_N[m_-(N)]}<\tau_{ \Ss_N[\{m_+(N), m_{\varepsilon_N}\}]} \,\Big| \,W_N\left(m_{\varepsilon_N}-\tfrac{2}{N}\right)\right)\\
			& = \PP_{\si}\left(W_N\left(m_{\varepsilon_N}+\tfrac{2}{N}\right)\right) \sum_{\si' \in \Ss_N\left[m_{\varepsilon_N}+\tfrac{2}{N}\right],\, \si \sim \si'} \PP_{\si'}\left(\tau_{\Ss_N[m_-(N)]}<\tau_{ \Ss_N[\{m_+(N), m_{\varepsilon_N}\}]} \right)  \\
			& \qquad + \PP_{\si}\left(W_N\left(m_{\varepsilon_N}-\tfrac{2}{N}\right)\right) \sum_{\si' \in \Ss_N\left[m_{\varepsilon_N}-\tfrac{2}{N}\right],\, \si \sim \si' } \PP_{\si'}\left(\tau_{\Ss_N[m_-(N)]}<\tau_{ \Ss_N[\{m_+(N), m_{\varepsilon_N}\}]} \right).\\
		\end{split}
	\end{equation}
	The first term vanishes because all the probabilities in the sum are zero. Thus, we get the upper bound
	\begin{multline}
		\PP_{\si}\Big(\tau_{\Ss_N[m_-(N)]}<\tau_{ \Ss_N[\{m_+(N), m_{\varepsilon_N}\}]} \Big) 
		\\ 
		\leq \sum_{\si' \in \Ss_N\left[m_{\varepsilon_N}-\tfrac{2}{N}\right], \si \sim \si'} \PP_{\si'}\Big(\tau_{\Ss_N[m_-(N)]}<\tau_{ \Ss_N[\{m_+(N), m_{\varepsilon_N}\}]} \Big).
	\end{multline}
	
	Using first Lemma~\ref{lem:hitting6.3} which gives bounds uniform in $\si'$, 
	we obtain
	\begin{equation}
		\begin{split}
			&	\PP_{\si}\left(\tau_{\Ss_N[m_-(N)]}<\tau_{ \Ss_N[\{m_+(N), m_{\varepsilon_N}\}]} \right)\\
			& \leq  \sum_{\si' \in \Ss_N\left[m_{\varepsilon_N}-\tfrac{2}{N}\right], \si \sim \si'} \exp\Big(-\beta N (1-\gamma)\left[f_{\beta}(m_-) + 
			\delta_N-f_{\beta}(m(\si'))\right]\Big) \\
			& =N \frac{1+ m_{\varepsilon_N}}{2}  \exp\Big(-\beta N (1-\gamma)\left[f_{\beta}(m_-) +  \delta_N-f_{\beta}\left(m_{\varepsilon_N}-\tfrac{2}{N}\right)\right]\Big) \\
			&= \exp\left(-\beta N (1-\gamma)\left[f_{\beta}(m_-) +  \delta_N-f_{\beta}\left(m_{\varepsilon_N}-\tfrac{2}{N}\right)  - \tfrac{\log N +O(1)}{\beta N (1-\gamma)}\right]\right)\\
			&\leq \exp\Big(-\beta N (1-\gamma)\left[f_{\beta}(m_-) +  \delta_N-f_{\beta}\left(m_{\varepsilon_N}-\tfrac{2}{N}\right)  - \varepsilon_N \right]\Big),\\
		\end{split}
	\end{equation}
	where in the last inequality we used that, for $N$ large enough, $\tfrac{\log N +O(1)}{\beta N (1-\gamma)}\leq \varepsilon_N$ (which converges to $\varepsilon>0$, see Step 4 in Section \ref{sec:notat_delta_theta}).
\end{proof}

In the proofs of Lemmas~\ref{lem:hittingP} and \ref{lem:for_mTheta}, we use
the following fact.
\begin{lemma} \label{lem:hitting6.3}
	For $\si \in \Ss_N\big[[m_{\delta_N},m_{\varepsilon_N})\big]$, for $N$ sufficiently large and any $\gamma \in (0,1)$,   
	\begin{equation}
		\PP_{\si}\left(\tau_{\Ss_N[m_-(N)]}<\tau_{ \Ss_N[\{m_+(N), m_{\varepsilon_N}\}]} \right) \leq \exp\Big(-\beta N (1-\gamma)\left[f_{\beta}(m_-) + \delta_N-f_{\beta}(m(\si))\right]\Big).
	\end{equation}	
\end{lemma}

\begin{proof}
	For $\si \in \Ss_N\big[[m_{\delta_N},m_{\varepsilon_N})\big]$, 
	\begin{equation}\label{eq:uppbound}
		\PP_{\si}\Big(\tau_{\Ss_N[m_-(N)]}<\tau_{ \Ss_N[\{m_+(N), m_{\varepsilon_N}\}]} \Big) \leq \PP_{\si}\Big(\tau_{\Ss_N[m_{\delta_N}]}<\tau_{ \Ss_N[\{m_+(N), m_{\varepsilon_N}\}]} \Big),
	\end{equation}	
	being $m_-(N)<m^*<m_{\delta_N}<m_{\varepsilon_N}<m_+(N)$, for $N$ sufficiently large. 	
	Therefore, we focus on finding an upper bound on the right hand side of \eqref{eq:uppbound}.
	
	Assume that there exists a function $\psi$ super-harmonic in $\Ss_N\big[[m_{\delta_N},m_{\varepsilon_N})\big]$. 	
	As a consequence, $0> L \psi (\si) = \frac{\partial}{\partial t}\EE_{\si} \left[ \psi (\sigma(t))\right]$. This implies $\EE_{\si}\left[\psi(\si(t))\right]\leq \EE_{\si}\left[\psi(\si(s))\right]$, for all $s<t$. Take $s=0$, and $\si(0)=\si$, therefore $\EE_{\si}\left[\psi(\si(t))\right]\leq \psi(\si)$, for all $t>0$. Thus, $\psi(\si(t))$ is a super-martingale. 
	For the integrable stopping time $T=\tau_{\Ss_N[m_{\delta_N}]} \wedge \tau_{ \Ss_N[\{m_+(N), m_{\varepsilon_N}\}]} $, we  use Doob's optional stopping theorem for super-martingales 
	to show that, for all $\si$ in the domain $\Ss_N\big[[m_{\delta_N},m_{\varepsilon_N})\big]$ of $\psi$, $\EE_{\si}\left[\psi(\si(T))\right]\leq \psi(\si)$.
	Therefore,
	\begin{equation}
		\psi(\si) \geq \EE_{\si}\big[\psi(\si(T))\big] 
		\geq \min_{\si' \in \Ss_N[m_{\delta_N}]} \psi(\si') \PP_{\si}\Big(\tau_{\Ss_N[m_{\delta_N}]}<\tau_{ \Ss_N[\{m_+(N), m_{\varepsilon_N}\}]} \Big),
	\end{equation} 
	which implies that
	\begin{equation}\label{eq:plug_psi}
		\begin{split}
			\PP_{\si}\Big(\tau_{\Ss_N[m_{\delta_N}]}<\tau_{ \Ss_N[\{m_+(N), m_{\varepsilon_N}\}]} \Big) &\leq \frac{\psi(\si)}{\min_{\si' \in \Ss_N[m_{\delta_N}]} \psi(\si') }.
		\end{split}
	\end{equation} 
	For a suitably chosen $\psi$ the latter inequality will yield the desired upper bound.
	Now we are left with the choice of a suitable $\psi: \Ss_N \to \RR$ such that $L\psi(x) < 0$, for all $x \in \Ss_N\big[[m_{\delta_N},m_{\varepsilon_N})\big]$. We define a function $\psi$ which depends on a parameter $\gamma \in (0,1)$ and is 
	constant on fixed magnetisation sets, i.e,  for all $\si \in \Ss_N$, 
	\begin{equation}\label{eq:def_psi}
		\psi(\si)=\phi(m(\si)),
	\end{equation}
	where $\phi: [-1,1] \to \RR$ is defined by 
	\begin{equation}\label{eq:defphi}
		\phi(m)= \exp\big(\beta N \,(1-\gamma) f_{\beta}(m)\big).
	\end{equation}
	Our choice of $\psi$ is similar to the one used by Bianchi, Bovier and Ioffe in~\cite[Proposition 6.4]{BBI09}. The choice of $\gamma$ is relevant in \eqref{eq:gammaConstr}.
	
	We claim and prove later in Lemma~\ref{lem:superHarm} that $\psi$ is super-harmonic in $\Ss_N\big[[m_{\delta_N},m_{\varepsilon_N})\big]$. Therefore, we conclude the proof by inserting  $\psi$ in \eqref{eq:plug_psi} and obtaining
	\begin{equation}
		\begin{split}
			\PP_{\si}\left(\tau_{\Ss_N[m_{\delta_N}]}<\tau_{ \Ss_N[\{m_+(N),\, m_{\varepsilon_N}\}]} \right) &\leq \frac{\exp\big(\beta N (1-\gamma) f_{\beta} (m(\si))\big)}{\min_{\si' \in \Ss_N[m_{\delta_N}]} \exp\big(\beta N (1-\gamma) f_{\beta} (m(\si'))\big) }\\
			&= \exp\left(\beta N (1-\gamma) \left[f_{\beta} (m(\si))-f_{\beta} (m_{\delta_N})\right]\right)\\
			&=   \exp\left(-\beta N (1-\gamma) \left[f_{\beta} (m_-) + \delta_N -f_{\beta} (m(\si))\right]\right),
		\end{split}
	\end{equation} 
	where we used the definition of $m_{\delta_N}$ (see Section~\ref{sec:notat_delta_theta}). 
\end{proof}

We are now left with the proof of the super-harmonicity of $\psi$, which is used in the proof of Lemma~\ref{lem:hitting6.3}.
\begin{lemma}\label{lem:superHarm}
	$\psi$ defined in \eqref{eq:def_psi} is super-harmonic in $\Ss_N\big[[m_{\delta_N},m_{\varepsilon_N})\big]$.
\end{lemma}
\begin{proof}
	We have to prove that $L\psi(x)<0$, for all $x \in \Ss_N\big[[m_{\delta_N},m_{\varepsilon_N})\big]$. Set $\bar{m}= m(x)$.
	As usual, we try to rewrite the terms appearing in the expression for $L\psi(x)$ in terms of their mean-field version. 
	
	\begin{equation}\label{eq:Lpsi}
		\begin{split}
			L\psi(x) &= \sum_{y \in \Ss_N } p(x,y) [\psi(y)-\psi(x)] \\
			&= \frac{1}{N}\sum_{y\in  \Ss_N} \mathds{1}_{y \sim x }\exp\big(-\beta [H(y)-H(x)]_+\big)  \\
			&\qquad\times\Big[\exp\left(\beta (1-\gamma)  N f_{\beta}(m(y))\right)- \exp\left(\beta(1-\gamma)  N f_{\beta}(m(x))\right)\Big] \\
			&=  \frac{1}{N}\sum_{m \in  \Gamma_N }\exp\big(-\beta N [E(m)-E(\bar{m})]_+\big)  \\
			&\qquad\times\Big[\exp\big(\beta N (1-\gamma)  f_{\beta}(m)\big)-
			\exp\big(\beta N(1-\gamma)  f_{\beta}(\bar{m})\big)\Big]\\
			& \qquad \qquad  \times \sum_{y : m(y)=m} \mathds{1}_{x \sim y} 
			\frac{\exp\big(-\beta [H(y)-H(x)]_+\big)}{\exp\big(-\beta N [E(m)-E(\bar{m})]_+\big)}\\
			&\leq\sum_{m \in  \Gamma_N } \exp\big(-\beta N [E(m)-E(\bar{m})]_+\big) \, \phi(\bar{m})  \,
			\Big[\exp\left(\beta N(1-\gamma) [f_{\beta}(m) - f_{\beta}(\bar{m})]\right) - 1\Big] \\
			& \qquad \qquad \times \eee^{2\beta} \left[ \frac{1+\bar{m}}{2}\mathds{1}_{\bar{m}-\frac{2}{N}}(m) +  \frac{1-\bar{m}}{2}\mathds{1}_{\bar{m}+\frac{2}{N}}(m) \right], 
		\end{split}
	\end{equation}
	where $\phi$ is defined in \eqref{eq:defphi} and we used the upper bound $\exp(2\beta)$ on $G(\sigma, m')$ as in the proof of the upper bound on capacity 
	(see \eqref{eq:ratio_trans+}, \eqref{eq:ratio_trans-}).
	
	Now, recalling definition \eqref{eq:rates_r}, we use the following notation
	\begin{equation}\label{r+}
		\begin{split}
			r_+=\tilde{r}_N\left(\bar{m},\bar{m}+\tfrac{2}{N}\right) 
			& =\exp\left(-2\beta \left[-\frac{1}{N} -(\bar{m}+h)\right]_+ \right) \frac{1-\bar{m}}{2},\\
		\end{split}
	\end{equation}
	\begin{equation}\label{r-}
		\begin{split}
			r_-=\tilde{r}_N\left(\bar{m},\bar{m}-\tfrac{2}{N}\right)  
			&= \exp \left(-2\beta \left[-\frac{1}{N} +\bar{m}+h\right]_+ \right) \frac{1+\bar{m}}{2},
		\end{split}
	\end{equation}
	and, for all $m \in \Gamma_N \setminus \{1\}$,
	\begin{equation}\label{eq:def_g}
		g(m)=\frac{N}{2} \left[f_{\beta}\left(m+\tfrac{2}{N}\right)-f_{\beta}(m) \right]. 
	\end{equation}
	Therefore, we can rewrite \eqref{eq:Lpsi} as
	\begin{equation}
		\begin{split}
			L\psi(x)&\leq  \eee^{2\beta} \,\phi(\bar{m}) \,r_+ \,\Big[\exp\big(2\beta(1-\gamma) g(\bar{m})\big) - 1\Big] \\
			&\qquad+ \eee^{2\beta}\, \phi(\bar{m}) \,r_-\, \Big[\exp\left(-2\beta(1-\gamma) g\left(\bar{m}-\tfrac{2}{N}\right)\right) - 1\Big]\\
			&= \eee^{2\beta} \,\phi(\bar{m})\, r_+\, G_+,
		\end{split}
	\end{equation}
	
	where
	\begin{equation}\label{eq:G+}
		\begin{split}
			G_+  = \Big(\mathrm{e}^{2\beta(1-\gamma)  g(\bar{m})} - 1\Big) 
			+\frac{r_- }{r_+} \left(\mathrm{e}^{-2\beta(1-\gamma)  g\left(\bar{m}-\tfrac{2}{N}\right)} - 1\right). 
		\end{split}
	\end{equation}
	
	Being $\eee^{2\beta} $, $\phi(\bar{m})$ and $r_+$ positive, we have only  to show that $G_+<0$. First  
	we notice that 
	\begin{equation}\label{eq:gm}
		g(m)=- m  -h + \frac{1}{\beta} I'(m) +O\left(\tfrac{1}{N}\right)
	\end{equation}
	and similarly 
	\begin{equation}\label{eq:gm-}
		g\left(m-\tfrac{2}{N}\right)=- m  -h + \frac{1}{\beta} I'(m) +O\left(\tfrac{1}{N}\right).
	\end{equation}
	Therefore,
	\begin{equation}\label{eq:g-g}
		\begin{split}
			g(m)-g\left(m-\tfrac{2}{N}\right)&=  O\left(\tfrac{1}{N}\right).
		\end{split}
	\end{equation}
	
	Then, since $I'(m)= \frac{1}{2}\log \Big(\frac{1+m}{1-m}\Big)$ (see \eqref{eq:limIN}), and using \eqref{eq:gm-} we have
	\begin{equation}\label{eq:r/r}
		\begin{split}
			\frac{r_-}{r_+}&= \frac{1+\bar{m}}{1-\bar{m}}  \frac{\exp \left(2\beta \left[-\frac{1}{N} -(\bar{m}+h)\right]_+  \right)}{\exp \left(2\beta \left[-\frac{1}{N} +\bar{m}+h\right]_+  \right)} \\
			&= \frac{1+\bar{m}}{1-\bar{m}} \exp \left(-2\beta (\bar{m}+h)  \right) \left(1+O\left(\tfrac{1}{N}\right)\right)\\
			&= \exp \big( 2 I'(\bar{m}) -2\beta (\bar{m}+h)  \big) \left(1+O\left(\tfrac{1}{N}\right)\right)\\
			& = \exp \left(2 \beta \left[ g\left(\bar{m}-\tfrac{2}{N}\right) + \bar{m}  +h +O\left(\tfrac{1}{N}\right) \right] -2\beta (\bar{m}+h)  \right)\left(1+O\left(\tfrac{1}{N}\right)\right) \\
			&= \exp\left(2 \beta \,g\left(\bar{m}-\tfrac{2}{N}\right) \right) \left(1+O\left(\tfrac{1}{N}\right)\right). 	
		\end{split}
	\end{equation}
	Therefore,  rearranging \eqref{eq:G+} and then using  \eqref{eq:g-g} and \eqref{eq:r/r}, we obtain
	\begin{equation}\label{eq:lastG+}
		\begin{split}
			G_+&= \Big[\exp\big(2\beta\,(1-\gamma) \,g(\bar{m})\big) - 1\Big]  \left[1- \frac{r_- }{r_+}\exp\left(-2\beta\, (1-\gamma)\,g\left(\bar{m}-\tfrac{2}{N}\right)\right)
			\right]  \\
			&\quad + \frac{r_- }{r_+} \Big[\exp\left(2\beta\,(1-\gamma)  \left[g(\bar{m})-g\left(\bar{m}-\tfrac{2}{N}\right)\right]\right) -1  \Big]\\
			&= \Big[\exp\big(2\beta\,(1-\gamma)\, g(\bar{m})\big) - 1\Big] \,
			\Big[1- \exp\left(2\beta\,\gamma \,g\left(\bar{m}-\tfrac{2}{N}\right) \right)\left(1+O\left(\tfrac{1}{N}\right)\right) \Big]\\
			&\quad + \frac{r_- }{r_+} \Big[ 
			\exp\left(2 \beta  \,(1-\gamma) \,O\left(\tfrac{1}{N}\right) \right) -1  \Big].\\
		\end{split}
	\end{equation}	
	Notice that, for every $m \in [m_{\delta_N},m_{\varepsilon_N}) \subset [m^*,m_+)$, $g(m)$ is negative, being $f_{\beta}$ strictly decreasing in $ [m^*,m_+)$. 
	As a consequence, $\mathrm{e}^{2\beta(1-\gamma) g(\bar{m})} - 1<0$. Furthermore, for $N$ sufficiently large, 
	$1- \eee^{2\beta \gamma g(\bar{m}-\frac{2}{N})}
	\left(1+ O\left(\frac{1}{N}\right)\right)>0$, implying that the first term in \eqref{eq:lastG+} is negative.
	
	Moreover, $\frac{r_-}{r_+}\geq 0$ is uniformly bounded from above, for $N$ sufficiently large. Therefore, since  $\beta$ is finite, $\gamma \in (0,1)$ and the term $\left[
	\exp\left(2 \beta  (1-\gamma) \, O\left(\frac{1}{N}\right)\right)-1 \right]$ is positive but converging to zero as $N$ grows to infinity, the second term in \eqref{eq:lastG+} is negligible.

	Therefore, for $N$ sufficiently large, $G_+$ is negative, concluding the proof.
	
\end{proof}

\subsection{Lower bound on the harmonic sum} \label{sec:lowerHarm}
In this section we provide the main ideas to prove the second part of Theorem~\ref{thm:upperHarm}, namely the lower bound 
on the harmonic sum in \eqref{eq:harmonicsum}.

\begin{proof}[Proof of Theorem~\ref{thm:upperHarm}. Lower bound.]
	
	The proof is very similar to the proof of the upper bound we gave in Section~\ref{sec:upperHarm}, therefore we omit the details. 
	The main contribution is given once again by the sum on $\Ss_N\left[U_{\delta,N}(m_-)\right]$.
	
	We have, 
	\begin{equation}\label{eq:sumUpLower}
		\begin{split}
			&\sum_{\si \in \Ss_N} \mu_{\beta,N}(\si) \,  h^N_{m_-,m_+}(\si) \geq  \sum_{\si \in \Ss_N[U_{\delta,N}(m_-)] } \mu_{\beta,N}(\si)  \,  h^N_{m_-,m_+}(\si) \\
			& =  \sum_{\si \in \Ss_N[U_{\delta,N}(m_-)]} \mu_{\beta,N}(\si)  -\sum_{\si \in \Ss_N[U_{\delta,N}(m_-)]} \mu_{\beta,N}(\si)  \big(1-h^N_{m_-,m_+}(\si)\big)\\
			& \geq \sum_{m \in U_{\delta,N}(m_-)\setminus [-1,m_-(N))} \Qq_{\beta,N}(m) \\
			& \qquad  -\sum_{\si \in \Ss_N[U_{\delta,N}(m_-)\setminus [-1,m_-(N))]} \mu_{\beta,N} (\si) \,\PP_{\si}\Big(\tau_{\Ss_N[m_+(N)]} < \tau_{\Ss_N[m_-(N)]}\Big).	\\
		\end{split}
	\end{equation} 
	
	The first term, i.e. the sum on the mesoscopic measure $\Qq_{\beta,N}$, gives the main contribution. 
	This sum can be estimated from below using the lower bound in Corollary~\ref{cor:theBound}, obtaining 
	a lower bound similar to the second upper bound in Corollary~\ref{cor:sumQ}  and applying the saddle point method as in \eqref{eq:part1}. 
	More precisely, using notation \eqref{eq:Psnotation}, we have the following lower bound  for $s>0$:
	\begin{equation}\label{eq:lower_Q-}
		\sum_{m \in U_{\delta,N}(m_-)\setminus [-1,m_-(N))} \Qq_{\beta,N}(m) \overset{P(s)}{\geq}  \frac{\eee^{\kappa-s}\, \exp\big(-\beta N f_{\beta} (m_-)\big)}{Z_{\beta,N} \sqrt{\left(1-m_-^2\right)\, \beta  \,f''_{\beta} (m_-) }}\,(1+o(1)).
	\end{equation}
	
	The second term in \eqref{eq:sumUpLower}, appearing with a negative sign in front, 
	is estimated via an upper bound, obtaining
	\begin{multline}\label{eq:lowerH_main}
		\sum_{\si \in \Ss_N[U_{\delta,N}(m_-)\setminus [-1,m_-(N))]} \mu_{\beta,N} (\si) \,\PP_{\si}\Big(\tau_{\Ss_N[m_+(N)]} < \tau_{\Ss_N[m_-(N)]}\Big) \\
		\leq  \frac{\eee^{s+\alpha} \exp\big({-\beta N f_{\beta} (m_-) }\big)} {Z_{\beta,N}} \sqrt{\frac{2}{\pi \left(1-m_+^2\right)}}   \, \eee^{-\beta N c} (1+o(1)),
	\end{multline}
	which is negligible compared to the right hand side of \eqref{eq:lower_Q-}, concluding the proof.
	
	We omit the proof of \eqref{eq:lowerH_main} being it again technical and very similar to the proof of the upper bound \eqref{eq:upper_part4} in Part 4 of Section~\ref{sec:upperHarm}. An analogue construction to the one given in Section~\ref{sec:notat_delta_theta} 
	and similar proofs to those in Section~\ref{sec:lemmas_hitting} are needed. The main difference consists in 
	restricting the analysis on a right neighbourhood of $m_-(N)$ instead of  a left neighbourhood of $m_+(N)$. 
	
\end{proof}	
	

	
\end{document}